\documentclass[11pt,a4paper]{amsart}
\usepackage[margin=3cm]{geometry}
\usepackage{amsmath,amssymb,amsthm,amsfonts}
\usepackage{enumitem}
\usepackage{xcolor}
\usepackage{comment}
\usepackage{todonotes}
\usepackage{microtype}
\usepackage{hyperref}
\usepackage{booktabs}
\usepackage{parskip}
\usepackage{graphicx}

\theoremstyle{plain}
\newtheorem{theorem}{Theorem}[section]
\newtheorem{lemma}[theorem]{Lemma}
\newtheorem{corollary}[theorem]{Corollary}
\newtheorem{proposition}[theorem]{Proposition}

\theoremstyle{definition}
\newtheorem{example}[theorem]{Example}
\newtheorem{definition}[theorem]{Definition}
\newtheorem{notation}[theorem]{Notation}

\theoremstyle{remark}
\newtheorem{remark}[theorem]{Remark}

\newcommand{\subsectionstart}{\leavevmode\vadjust{\vspace{0.5em}}\noindent}

\newcommand{\C}{\mathbb{C}}
\newcommand{\R}{\mathbb{R}}
\newcommand{\RP}{\mathbb{RP}}
\newcommand{\CP}{\mathbb{CP}}

\newcommand{\mA}{\mathcal{A}}
\newcommand{\mB}{\mathcal{B}}
\newcommand{\mG}{\mathcal{G}}
\newcommand{\mH}{\mathcal{H}}
\newcommand{\mF}{\mathcal{F}}

\newcommand{\mS}{\mathcal{S}}
\newcommand{\mT}{\mathcal{T}}
\newcommand{\mV}{\mathcal{V}}

\DeclareMathOperator{\Id}{Id}
\DeclareMathOperator{\tr}{tr}

\DeclareMathOperator{\GL}{GL}
\DeclareMathOperator{\Span}{span}

\DeclareMathOperator{\Sym}{Sym}
\DeclareMathOperator{\relin}{relin}

\newcommand{\inn}{(\cdot , \cdot)}

\newcommand{\msv}{\mathcal{M}_{[\mathcal{A}]}}

\title[\(\vee\)-systems and PK arrangements]{Euclidean $\vee$-systems and real PK arrangements}

\author{M. de Borbon}
\address{Loughborough University, Loughborough LE11 3TU, UK}
\email{m.de-borbon@lboro.ac.uk}

\author{D. Panov}
\address{King's College London,
Strand London WC2R 2LS, UK}
\email{dmitri.panov@kcl.ac.uk}

\author{A.P. Veselov}
\address{Loughborough University, Loughborough LE11 3TU, UK}
\email{A.P.Veselov@lboro.ac.uk}

\begin{document}

\begin{abstract}
We establish a correspondence between two structures arising in the geometry of hyperplane arrangements:
Euclidean \(\vee\)-systems and real polyhedral K\"ahler (PK) arrangements. 
We prove that every irreducible Euclidean \(\vee\)-system determines a real PK arrangement, and conversely that every real PK arrangement arises this way. 
As a consequence, we show that, up to equivalence, there are exactly three irreducible rank-three Euclidean \(\vee\)-systems whose vectors have equal length; their arrangements are the mirrors of the reflection groups of the regular tetrahedron, cube, and icosahedron.
The correspondence also yields a description of the moduli space of Euclidean \(\vee\)-systems in a fixed projective class: it is homeomorphic to the relative interior of a polytope.
We also give a direct proof that the hyperplane arrangement associated with a Euclidean \(\vee\)-system is simplicial. 
Among the currently known simplicial line arrangements, we identify precisely those that arise from \(\vee\)-systems.
As a consequence, we prove that the Schreiber--Veselov catalog is complete for irreducible rank-three Euclidean \(\vee\)-systems with at most \(27\) vectors.
\end{abstract}

\maketitle

\section{Introduction}

A major development in differential geometry of 1990s was the creation of the theory of Frobenius manifolds initiated by Dubrovin \cite{D1}. It was inspired by the work of theoretical physicists Witten--Dijkgraaf--Verlinde--Verlinde \cite{DVV} and unified many areas of mathematics and mathematical physics from topological quantum field theory to algebraic and enumerative geometry, singularity theory and integrable systems. 

The \emph{associativity equations} (known also as {\it WDVV equations}) play a central role in the theory of Frobenius manifolds. Remarkably, similar equations appeared also in Seiberg-Witten work on $N=2$ SUSY Yang-Mills theory, which led to the extension of the original Dubrovin's framework \cite{MMM}.

The notion of $\vee$-systems was introduced by one of the authors \cite{veselov1} in relation with a special class of solutions of the generalized WDVV equations. These systems are certain finite collections of covectors, which can be viewed as far-reaching generalizations of the Coxeter root systems (see Section \ref{sec:vsystems}). They were studied in more detail in \cite{CV,FV1,feigin-veselov-geometryVsystem,feigin-veselov-holonomyliealgebras,schreiberveselov}, 
but the classification of $\vee$-systems remained one of the most challenging problems in this area. In this paper we make an important progress in this direction, which we now explain.

We introduce a geometric approach to the classification problem for Euclidean \(\vee\)-systems. More precisely, we relate the real \(\vee\)-systems to the theory of polyhedral K\"ahler manifolds \cite{panov, dbp-cpn} by establishing a precise correspondence between irreducible Euclidean \(\vee\)-systems and real PK arrangements. The main conceptual input is the Hirzebruch quadratic form, studied first for line arrangements in \cite{hirzebruch,panov} and later defined for hyperplane arrangements in arbitrary dimension in \cite{dbp-miyaokayau}. We adapt the variational approach of \cite{dbp-hirzebruch}, based on moment maps and the Kempf--Ness theorem, to the real setting and relate the vanishing of the Hirzebruch quadratic form to the existence of \(\vee\)-systems.

As a result, we obtain a new existence and uniqueness mechanism for Euclidean \(\vee\)-systems. This method also describes all possible solutions in a fixed projective class: the moduli space is the intersection of the relative interior of the matroid base polytope with the kernel of the Hirzebruch quadratic form. Consequently, the existence and classification of Euclidean \(\vee\)-systems can be expressed in terms of combinatorial data attached to hyperplane arrangements.

We also prove a structural property of Euclidean \(\vee\)-systems: their associated hyperplane arrangements are simplicial. 
Using the correspondence with PK arrangements 
we show that the Schreiber--Veselov rank-three catalog \cite{schreiberveselov} is complete among all known simplicial line arrangements.
Combining this together with the classification of simplicial arrangements with at most \(27\) lines \cite{cuntz}, we prove that the Schreiber--Veselov catalog is complete for irreducible Euclidean \(\vee\)-systems with at most \(27\) vectors. 


\subsection{Main results}

Our first two main results provide a correspondence between Euclidean \(\vee\)-systems and real PK arrangements, up to natural equivalences. 

First, we show that an irreducible Euclidean \(\vee\)-system defines a real PK arrangement. 

\begin{theorem}\label{thm:fromVtoPK}
    Let \(\mA = \{\alpha_i\}_{i=1}^n\) be an irreducible Euclidean \(\vee\)-system in \(\R^d\) with \(d \geq 2\).
    Let \((\Delta_{\mA}, r_\mA)\) with \(\Delta_{\mA} = \{\Delta_1, \ldots, \Delta_n\}\) and \(r_{\mA} = (r_1, \ldots, r_n)\)  be the associated weighted arrangement consisting of hyperplanes \(\Delta_i = \alpha_i^{\perp}\) and weights  \(r_i = |\alpha_i|^2\).
    Then  \((\Delta_{\mA}, r_{\mA})\) is a real PK arrangement.
\end{theorem}

Conversely, we show that every real PK arrangement comes from an irreducible Euclidean \(\vee\)-system.

\begin{theorem}\label{thm:fromPKtoV}
    Let \((\mH, w)\) be a real PK arrangement in \(\R^d\) with \(d \geq 2\). Then there exists an irreducible Euclidean \(\vee\)-system \(\mA\) such that the associated weighted arrangement \((\Delta_{\mA}, r_{\mA})\) satisfies the following properties: \textup{(i)} \(\Delta_{\mA}\) is linearly equivalent to \(\mH\); and \textup{(ii)} \(r_{\mA} = w\). Moreover, \(\mA\) is unique up to orthogonal projective equivalence.
\end{theorem}

The above two theorems combined together give us the following.

\begin{corollary}\label{cor:bijection}
    The map \(\mA \mapsto (\Delta_{\mA}, r_{\mA})\) defines a bijection between irreducible Euclidean \(\vee\)-systems up to orthogonal projective equivalence \textup{(Definition~\ref{def:orthogonally-proj-equiv})} and real PK arrangements up to linear equivalence \textup{(Definition~\ref{def:linear-equivalence})}.
\end{corollary}

From the above correspondence, combined with \cite{panov-real-hir}, we obtain the following classification result for rank-three \(\vee\)-systems.

\begin{corollary}\label{cor:classification}
Up to equivalence, there exist exactly \(3\) irreducible Euclidean \(\vee\)-systems of rank \(3\) whose vectors have the same length. The corresponding arrangements are the mirrors of the reflection groups of the regular tetrahedron, cube, and icosahedron.
\end{corollary}

We also derive a description of the moduli space of Euclidean \(\vee\)-systems in a fixed projective equivalence class in terms of combinatorial data.

\begin{theorem}\label{thm:modulispace}
    The moduli space \(\msv\) of irreducible Euclidean \(\vee\)-systems within a fixed projective equivalence class \([\mA]\) up to orthogonal projective equivalence is homeomorphic to the interior of a polytope. More precisely, 
    \[
    \msv \cong \relin(P(\mA)) \cap \ker Q \,,
    \]
    where \(\relin(P(\mA))\) is the relative interior of the base polytope \(P(\mA)\) and
    \(\ker Q\) denotes the kernel of the Hirzebruch quadratic form \(Q\) of \(\Delta_{\mA}\).
\end{theorem}

In particular, since the Hirzebruch quadratic form is entirely determined by the matroid polytope, it follows that the property of a vector configuration being projectively equivalent to a \(\vee\)-system (i.e. \(\msv\) being non-empty) depends only on the matroid polytope of \([\mA]\).

Next, we relate \(\vee\)-systems to simplicial arrangements.

\begin{theorem}\label{thm:simplicial}
    Let \(\mA\) be a Euclidean \(\vee\)-system. Then the associated hyperplane arrangement \(\Delta_{\mA}\) is simplicial.
\end{theorem}

Finally, we derive the following classification result.

\begin{theorem}\label{thm:classification3d}
    Let \(\mA\) be an irreducible Euclidean \(\vee\)-system in \(\R^3\). 
    If \(|\mA| \leq 27\), then \(\mA\) is one of the \(\vee\)-systems listed in Schreiber-Veselov \cite{schreiberveselov}.
\end{theorem}

We would like to stress that our approach provides a fundamentally new route to the Schreiber–Veselov catalog. Namely, all known \(\vee\)-systems can be obtained from Gr\"unbaum’s original 1972 catalog of sporadic simplicial arrangements by solving, for each arrangement, a system of linear equations and inequalities.

\subsection{Outline of the paper}

Sections~\ref{sec:vsystems} and~\ref{sec:PK} recall the two structures that will be compared: Euclidean \(\vee\)-systems and real PK arrangements. 

Sections~\ref{sec:tf} and~\ref{sec:barthe} contain the main background tools. Section~\ref{sec:tf} reviews finite tight frames, which extend the concept of an orthonormal basis to evenly distributed vector sequences. Section~\ref{sec:barthe} recalls Barthe's theorem, which characterizes when a weighted projective class of vectors contains a normalized tight frame with prescribed squared lengths.

Section~\ref{sec:hqf} forms the  technical core of the paper. The main result of the section is Proposition~\ref{prop:hirineq} which characterizes \(\vee\)-systems, among the broader class of normalized tight frames, as extremizers of the Hirzebruch quadratic form. 


The correspondence between Euclidean \(\vee\)-systems and real PK arrangements is proved in Sections~\ref{sec:vtop} and~\ref{sec:ptov}. In Section~\ref{sec:vtop} we show that every irreducible Euclidean \(\vee\)-system gives a real PK arrangement (Theorem \ref{thm:fromVtoPK}). In Section~\ref{sec:ptov} we prove the converse (Theorem \ref{thm:fromPKtoV}), using Barthe's theorem to construct the required normalized tight frame, and then applying the criterion from Proposition~\ref{prop:hirineq}.

Section~\ref{sec:corrolaries} gives a short proof of Corollaries \ref{cor:bijection} and \ref{cor:classification}.

Section~\ref{sec:moduli} applies the correspondence given by Theorems~\ref{thm:fromVtoPK} and~\ref{thm:fromPKtoV} to moduli. It identifies the moduli space of Euclidean \(\vee\)-systems in a fixed projective class with the intersection of the relative interior of the matroid base polytope with the kernel of the Hirzebruch quadratic form.

In Section~\ref{sec:simplicial} we provide a direct and elementary proof that the hyperplane arrangement associated with any Euclidean \(\vee\)-system is simplicial.
In Section~\ref{sec:3dclassification} we combine this with the classification of simplicial line arrangements with at most \(27\) lines, showing that the Schreiber--Veselov catalog is complete in this range. The appendix~\ref{app:tables} records the known simplicial line arrangements which come from \(\vee\)-systems.

\subsection*{Acknowledgments.}
 We would like to thank Misha Feigin for useful comments.  

\section{\(\vee\)-systems}\label{sec:vsystems}

We start with the definition of the $\vee$-systems in the (real) Euclidean form.
We should mention that in \cite{feigin-veselov-geometryVsystem} it was also introduced the notion of complex Euclidean $\vee$-systems, but in this paper we will consider only their real versions, leaving the investigation of the complex case for the future.

\begin{definition}[\cite{veselov1,veselov2,feigin-veselov-geometryVsystem,schreiberveselov}]\label{def:Vsystem}
	A {\it Euclidean \(\vee\)-system} in \(\R^d\) is a finite set 
	\[
	\mA = \{\alpha_1, \ldots, \alpha_n\} \subset \R^d \setminus \{0\}
	\] 
	of pairwise non-proportional vectors satisfying the following properties.
	\begin{enumerate}[label=\textup{(V\arabic*)}]
		\item \label{it:V1} For all \(x \in \R^d\):
		\begin{equation*}
		\sum_{i=1}^n (x, \alpha_i)^2 =|x|^2 \,,
		\end{equation*}
		where \(\inn\) denotes the standard Euclidean inner product.
		\item \label{it:V2} For every linear \(2\)-plane \(\Pi \subset \R^d\) with \(\left|\mA \cap \Pi \right| \geq 3\) and all \(x\in \Pi\):
		\begin{equation*}
		\sum_{i \,|\, \alpha_i \in \Pi} (x, \alpha_i)^2 = \nu(\Pi) \cdot |x|^2  \,,
		\end{equation*}
		for some $\nu(\Pi) \in \R$.
		
		\item \label{it:V3} For every linear \(2\)-plane \(\Pi \subset \R^d\) with \(\mA \cap \Pi = \{\alpha_i, \alpha_j\}\) for some \(i \neq j\):
		\begin{equation*}
		(\alpha_i, \alpha_j) = 0 \,.
		\end{equation*}
	\end{enumerate}
\end{definition}

The original definition of the $\vee$-systems in \cite{veselov1} is the $GL(V)$-invariant version of the Euclidean one and looks as follows. 

Let $V$ be a real vector space and  $\mathcal{A}\subset V^*$ be a
finite set of pairwise non-collinear vectors in the dual space $V^*$ (covectors) spanning $V^*$. To such a set one can
associate the following positive definite {\it canonical form} $G_{\mathcal A}$ on
$V$:
\begin{equation}
\label{GEC} 
G_{\mathcal A}(x,y)=\sum_{\alpha\in\mathcal{A}}\alpha(x)\alpha(y),\,\,  x,y\in V
\end{equation}
which establishes the isomorphism
$$
\varphi_{\mathcal A}: V \rightarrow V^*.
$$
Let $\alpha^\vee = \varphi_{\mathcal A}^{-1}(\alpha)$ be the corresponding inverse image of $\alpha \in \mathcal A.$
Note that because of the choice of the canonical form $\alpha^\vee$ is a complicated function of all $\alpha \in \mathcal A$.

The system $\mathcal{A}$ is called $\vee$-{\it system} if the following $\vee$-{\it conditions} 
\begin{equation}
\label{vee}
\sum\limits_{\beta \in \Pi \cap \mathcal {A}}
\beta(\alpha^\vee)\beta^\vee=\nu \alpha^{\vee}
\end{equation}
are satisfied for any $\alpha \in \mathcal{A}$ and any two-dimensional plane $\Pi \subset V^*$ containing $\alpha$ and some $\nu$, which may depend on $\Pi$ and $\alpha.$ If $\Pi$ contains more than 2 covectors then (\ref{vee}) imply that $\nu$ does not depend on $\alpha \in \Pi$ 
and 
\begin{equation}
\label{vee1} 
\sum\limits_{\beta \in \Pi \cap \mathcal {A}}
\beta^\vee \otimes \beta |_{\Pi} = \nu(\Pi) Id.
\end{equation}
If $\Pi$ contains only two covectors from $\mathcal A$, say $\alpha$ and $\beta,$ then (\ref{vee}) imply that 
\begin{equation}
\label{vee2} 
G_{\mathcal A}(\alpha^{\vee}, \beta ^{\vee})=0.
\end{equation}

The $\vee$-conditions are equivalent to the flatness of the {\it Dubrovin connection} on the complement to the corresponding hyperplane arrangement
\begin{equation}\label{veecon}
	\nabla^\vee_\xi=\partial_\xi - \kappa\sum_{\alpha\in {\mathcal A}}\frac{(\alpha,\xi)}{(\alpha,x)}\alpha \otimes \alpha. 
\end{equation}

The restriction of a $\vee$-system $\mathcal A$ to the  subspace defined by a subset $\mathcal B\subset \mathcal A$ is also a $\vee$-system (see \cite{FV1,feigin-veselov-geometryVsystem}),  so the $\vee$-systems can be considered as an extension of the class of Coxeter root systems closed under such operation. 

In \cite{feigin-veselov-holonomyliealgebras} it was shown that the  hyperplane arrangements corresponding to all known real $\vee$-systems can be obtained from the Coxeter arrangements and their restrictions (so, in particular, they are all simplicial) and it was conjectured that they are always hereditary-free in Saito's sense.

The examples of $\vee$-systems include all two-dimensional systems, 
Coxeter systems and the so-called deformed root systems \cite{veselov1}.
In particular, we have the following two families found in \cite{CV}:
\begin{equation}\label{a}
\mathfrak
{A}_n(c)=\{\sqrt{c_ic_j}(e_i-e_j),\quad 0\le i<j\le n\},
\end{equation}
\begin{equation}\label{bc}
\mathfrak{B}_n(c)= \left\{
\begin{array}{lll}
\sqrt{c_ic_j}(e_i \pm e_j)\,, &  1\le i<j\le n\,,
\\ \sqrt{2 c_i(c_i + c_0)}e_i\,, &  i=1,\ldots ,n\,.
\end{array}
\right.
\end{equation}
where $c_0, c_1,\dots, c_{n}$ are arbitrary positive
parameters. 

The full classification of $\vee$-systems is still an open problem. Since in dimension 2 we have no restrictions, the first interesting, and also crucial case is classification in dimension 3. All the 3D examples known so far were described in \cite{feigin-veselov-geometryVsystem} (see the schematic map in the Appendix) and cataloged in \cite{schreiberveselov}. In \cite{lechtenfeld_et_al} it was shown using computer that this list is complete for the $\vee$-systems with no more than 10 covectors. One of the results of our paper is the completeness of this list for the systems with up to 27 covectors. Corollary~\ref{cor:classification} settles the classification problem in the case of $\vee$-systems of rank \(3\) consisting of vectors of the same length.


To define the corresponding Euclidean $\vee$-system one can simply introduce the Euclidean structure on $V$ using the canonical form $G_{\mathcal A}(x,y)$, so that property $(V1)$ will be satisfied automatically. In other words, the Euclidean $\vee$-systems are the $\vee$-systems with the chosen canonical form.

There are also trigonometric and elliptic versions of $\vee$-systems introduced in \cite{feigin_trig} and \cite{strachan} in relation with the corresponding classes of solutions of the WDVV equation (see more recent results in \cite{alkadhem_feigin}, \cite{stedman_strachan}). The relation with polynomial solutions of the WDVV equation and Frobenius--Saito structure on the orbit spaces of Coxeter groups was explained by Dubrovin \cite{dubrovin_almost}.

\section{Real PK arrangements and polyhedral K\"ahler geometry}\label{sec:PK} 

We will now give the definition of real PK arrangements and will briefly recall the motivation for studying them, coming from polyhedral K\"ahler geometry.  

\begin{definition}[\cite{panov}, \cite{dbp-cpn}]\label{def:PK}
A \emph{real PK arrangement} in \(\R^d\) with \(d \geq 2\) is a pair \((\mH, w)\), where
\(\mH = \{H_1,\ldots,H_n\}\) is a finite set of pairwise distinct linear hyperplanes \(H_i \subset \R^{d}\) and \(w=(w_1,\ldots,w_n)\) is a vector in \(\R^n\) with 
\[
w_i\in(0,1) \quad \textup{for all } i \,,
\]
satisfying the following properties.
\begin{enumerate}[label=\textup{(PK\arabic*)}]
\item \label{it:PK1}
\[
\sum_{i=1}^n w_i = d .
\]

\item \label{it:PK2}
For every non-zero proper linear subspace \(L\subset \R^{d}\) obtained as the intersection of hyperplanes in \(\mH\), we have
\[
\sum_{i\,:\,L\subset H_i} w_i < r(L) \,,
\]
where \(r(L) = d - \dim L\) denotes the codimension of \(L\).

\item \label{it:PK3}
\[
Q(w_1, \ldots, w_n)=0 \,,
\]
where \(Q\) is the Hirzebruch quadratic form of \(\mH\) recalled below.
\end{enumerate}
\end{definition}

\begin{definition}[\cite{dbp-miyaokayau}]\label{def:hirqf}
The \emph{Hirzebruch quadratic form} of \(\mH\) is the homogeneous degree-\(2\) polynomial on \(\R^n\) given by
\begin{equation}\label{eq:hir}
Q(w_1, \ldots, w_n) = \sum_{S \in \mG} w_S^2 \,-\, \frac{1}{2} \cdot \sum_{i=1}^n (\sigma_i-1) \, w_i^2 \,-\, \frac{1}{2d} \cdot \left(\sum_{i=1}^{n} w_i\right)^2    \,,
\end{equation}
where: \(\mG\) is the set of codimension \(2\) linear subspaces \(S \subset \R^{d}\) that are the common intersection of \(3\) or more hyperplanes in \(\mH\)\,;
\(w_S\) for \(S \in \mG\) is the linear function given by
	\begin{equation*}
	2 \, w_S = \sum_{i \,|\, S \subset H_i} w_i \,;
	\end{equation*}
and \(\sigma_i\) is the number of elements \(S \in \mG\) with \(S \subset H_i\).
\end{definition}

\begin{remark} Real PK arrangements are a special case of complex PK arrangements, for which the definition is identical with the only difference that instead of real hyperplane arrangements we consider complex hyperplane arrangements. The complexation of a real PK arrangement is naturally a complex PK arrangement.
\end{remark}

The definition of complex PK arrangement has a topological origin and is motivated by the study of polyhedral K\"ahler manifolds introduced in \cite{panov}.
By definition, PK manifolds are even dimensional PL manifolds with a piecewise Euclidean structure and  unitary holonomy. The metric singularities happen along a collection of complex divisors on the manifold, such as a hyperplane arrangement in $\mathbb{CP}^{d-1}$. 
The conditions in the definition of PK arrangement encode the topological constraints that are necessary and sufficient for the existence of a polyhedral K\"ahler metric on \(\CP^{d-1}\) with prescribed cone angles \(2\pi(1-w_i)\) along the complex hyperplanes \(H_i \subset \CP^{d-1}\) as shown in \cite{panov} for \(d=3\) and \cite{dbp-cpn} for all \(d \geq 3\).
In this setup equations (\ref{it:PK1}--\ref{it:PK3}) follow from Gauss--Bonnet formulas applied to the singular metric. In complex dimensions $2$ and higher PK metrics are highly rigid, and it is expected that complements to PK arrangements in $\mathbb {CP}^{d-1}$ are $K(\pi,1)$ spaces. This was proven  for $d=3$ in \cite{panovpetrunin-ramification} and holds for (comlpexifications of) real PK arrangements by Theorem \ref{thm:simplicial}.

From a broader perspective, PK metrics fit naturally into the theory of K\"ahler--Einstein metrics with conical singularities along divisors \cite{donaldson, deborbon_spotti}. They occupy a role analogous to that of complex hyperbolic manifolds in negative curvature: they arise as equality cases of Miyaoka--Yau type inequalities,  \cite{panov, dbp-miyaokayau}.

\subsection{Linear equations for PK arrangements}\label{sec:linear-pk}

We next give an equivalent characterization of real PK arrangements in terms of linear equations and strict linear inequalities. This characterization is a higher dimensional generalization of the rank three case, derived in \cite[Section~5.3]{panov}.
The key input is the non-positivity of \(Q\) on the set of \emph{stable weights}, proved in \cite{dbp-miyaokayau,dbp-hirzebruch}. See also Proposition~\ref{prop:hirineq} below.

\begin{proposition}\label{prop:pk3-kernel}
Let \(w\in(0,1)^n\) satisfy~\ref{it:PK1} and~\ref{it:PK2}. Then condition~\ref{it:PK3} holds if and only if \(w\in\ker Q\), where \(\ker Q \subset \R^n\) is the linear subspace given by the kernel of the Hirzebruch quadratic form.
\end{proposition}

\begin{proof}
	Conditions \ref{it:PK1} and \ref{it:PK2} imply that \(w \in C^{\circ}\), where \(C^{\circ} \subset \R^n\) is the open polyhedral cone of stable weights \cite[Definition~6.15]{dbp-miyaokayau}. By \cite[Theorem~6.18]{dbp-miyaokayau}, \(Q \leq 0\) on \(C^{\circ}\). Hence, if \(Q(w)=0\), then \(w\) is a local maximum of \(Q\), and therefore \(dQ(w)=0\). Since \(Q\) is a quadratic form, this is equivalent to \(w\in\ker Q\). The converse is immediate.
\end{proof}

\begin{remark}
	The above proposition can also be proved using the results developed later in this paper. See Section~\ref{sec:pf-moduli-thm}.
\end{remark}

We give an explicit form of the linear condition \(w\in\ker Q\).
We use the same notation \(Q\) for the Hirzebruch quadratic form and for its associated symmetric matrix, so that
\[
Q(w)=w^tQw=\sum_{i,j=1}^n Q_{ij}w_iw_j.
\]
The coefficients \(Q_{ij}\) can be obtained by expanding \eqref{eq:hir}.
For an intersection \(S\) of hyperplanes in \(\mH\), let
\[
\operatorname{mult}(S)=\#\{i \mid S\subset H_i\}
\]
denote its multiplicity. Then
\[
Q_{ii}=\frac{2(d-1)-d\sigma_i}{4d},
\qquad
Q_{ij}=\begin{cases}
-\dfrac{1}{2d} & \textup{if } \operatorname{mult}(H_i\cap H_j)=2,\\[4pt]
\dfrac{d-2}{4d} & \textup{if } \operatorname{mult}(H_i\cap H_j)\geq 3.
\end{cases}
\]

\begin{definition}\label{def:Bmatrix}
Using the numbers \(\sigma_i\) from Definition~\ref{def:hirqf}, define the symmetric matrix \(B=(B_{ij})\in M_n(\mathbb{Z})\) by
\[
B_{ii}=\sigma_i-1
\]
and, for \(i\neq j\),
\[
B_{ij}=\begin{cases}
1 & \textup{if } \operatorname{mult}(H_i\cap H_j)=2,\\
0 & \textup{if } \operatorname{mult}(H_i\cap H_j)\geq 3.
\end{cases}
\]
\end{definition}

Let \(J\) be the \(n\times n\) matrix whose entries are all equal to \(1\). A direct check shows that
\begin{equation}\label{eq:Q-B}
4d\,Q=(d-2)J-dB.
\end{equation}
Let \(\mathbf{1}=(1,\ldots,1)^t\).
\begin{lemma}\label{lem:Q-B}
If \(w\) satisfies~\ref{it:PK1}, then
\begin{equation}\label{eq:kernel-B}
w\in\ker Q
\quad\Longleftrightarrow\quad
Bw=(d-2)\mathbf{1}.
\end{equation}
\end{lemma}

\begin{proof}
If \(w\) satisfies~\ref{it:PK1}, then \(Jw=d\mathbf{1}\). Hence \eqref{eq:Q-B} gives \eqref{eq:kernel-B}.
\end{proof}

Combining Proposition~\ref{prop:pk3-kernel} and Lemma~\ref{lem:Q-B} gives the following characterization.

\begin{corollary}\label{cor:linear-pk}
A pair \((\mH,w)\) is a real PK arrangement if and only if \(w \in S \cap \mathcal{U}\), where \(S\) is the affine subspace
\begin{equation}\label{eq:linear-pk-system}
Bw=(d-2)\mathbf{1},\qquad \sum_{i=1}^n w_i=d,
\end{equation}
and \(\mathcal{U}\) is the open set of \(w \in (0,1)^n\) which satisfy \ref{it:PK2}. 
\end{corollary}

\section{Tight frames}\label{sec:tf}

\(\vee\)-Systems are a particular class of vector configurations known as normalized tight frames.
This section collects the elementary facts from the theory of tight frames that will be used in the proof of the correspondence between PK arrangement and \(\vee\)-systems (Theorems~\ref{thm:fromVtoPK} and~\ref{thm:fromPKtoV}). We recall the definition of tight frames, their behavior under orthogonal projection, and the trace formula (Lemma~\ref{lem:traceformula}). The main consequence
needed later is the stability inequality (Proposition~\ref{prop:stabineq}) for normalized tight frames. Lemma~\ref{lem:traceformula} and Proposition~\ref{prop:stabineq} provide the linear-algebraic mechanism behind conditions~\ref{it:PK1} and~\ref{it:PK2} appearing in the Definition~\ref{def:PK} of real PK arrangements.

\subsection{Tight frames}

\subsectionstart

\noindent We work on \(\R^d\) whose elements are understood as column vectors. 
We equip \(\R^d\) with its standard Euclidean inner product \((x,y) = x^\top y\).
For a vector \(x \in \R^d\), we denote by \(x \otimes x\) the linear endomorphism of \(\R^d\) defined by
\begin{equation}
    (x \otimes x)(y) = (x,y)\, x \,.
\end{equation}
In matrix form, \(x \otimes x = x \, x^{\top}\).

\begin{definition}\label{def:tightframe}
	A finite sequence of vectors \((\alpha_i)_{i=1}^n\) in \(\R^d\) is a \emph{tight frame} for \(\R^d\) if 
	\begin{equation}\label{eq:tightframe}
	\sum_{i=1}^{n} \alpha_i \otimes \alpha_i = \lambda \cdot \Id
	\end{equation}
	for some positive scalar \(\lambda\). We say that \((\alpha_i)_{i=1}^n\) is a \emph{normalized tight frame}\footnote{Normalized tight frames are commonly known as \emph{Parseval frames}.} if \(\lambda=1\).
\end{definition}

\begin{remark}
	We do not impose any additional assumptions on the vectors \(\alpha_i\) in Definition~\ref{def:tightframe}. In particular, the sequence \((\alpha_i)_{i=1}^n\) is allowed to contain zero vectors as well as repeated elements. This is consistent with standard conventions in the literature; see, for example, \cite[Definition 2.1]{waldron}. 
	When the vectors \(\alpha_1,\dots,\alpha_n\) are pairwise distinct, we also say that the set
	\(\{\alpha_i\}_{i=1}^n\) is a (normalized) tight frame for \(\R^d\) whenever the sequence \((\alpha_i)_{i=1}^n\) is a (normalized) tight frame for \(\R^d\).
\end{remark}

\begin{lemma}\label{lem:span}
	If \((\alpha_i)_{i=1}^n\) is a tight frame for \(\R^d\) then the vectors \(\alpha_1, \ldots, \alpha_n\) span \(\R^d\).
\end{lemma}

\begin{proof}
	If \(x \in \R^d\) is such that \((x, \alpha_i) = 0\) for all \(i\) then Equation~\eqref{eq:tightframe} implies that \(x=0\).
\end{proof}

\begin{example}
	The standard basis \((e_i)_{i=1}^n\) of \(\R^n\) is a normalized tight frame for \(\R^n\).
\end{example}

An equivalent characterization of the tight frame condition is given by the following.

\begin{lemma}\label{lem:tfequiv}
    A sequence of vectors \((\alpha_i)_{i=1}^n\) in \(\R^d\) is a tight frame for \(\R^d\) if and only if there is some \(\lambda>0\) such that
    \begin{equation}\label{eq:tf2}
       \sum_{i=1}^n (x,\alpha_i)^2 = \lambda \,|x|^2 \quad \textup{for all } x \in \R^d \,. 
    \end{equation}
\end{lemma}

\begin{proof}
    Suppose that Equation~\eqref{eq:tightframe} holds. Then for every \(x \in \R^d\) we have \(\sum (x,\alpha_i) \alpha_i = \lambda x\). Taking the inner product with \(x\) gives Equation~\eqref{eq:tf2}. Conversely, suppose that Equation~\eqref{eq:tf2} holds.
    Let \(A\) and \(B\) be the endomorphisms of \(\R^d\) given by \(A = \sum \alpha_i \otimes \alpha_i\) and \(B = \lambda \Id\). By assumption, \((Ax,x)=(Bx,x)\) for all \(x \in \R^d\). Since \(A\) and \(B\) are symmetric, \(A=B\), i.e., Equation~\eqref{eq:tightframe} holds.
\end{proof}

\subsection{Orthogonal projection of a tight frame}

\begin{lemma}[{\cite[Exercise 2.6]{waldron}}]\label{lem:projS}
	Suppose that \((\alpha_i)_{i=1}^n\) is a normalized tight frame for \(\R^d\). Let \(S \subset \R^d\) be a linear subspace and denote by \(P_S\) the orthogonal projection onto \(S\). For each vector \(\alpha_i\) let  \(\beta_i = P_S(\alpha_i)\). Then \((\beta_i)_{i=1}^n\) is a normalized tight frame for \(S\).
\end{lemma}

\begin{proof}
	Let \(y \in S\), we want to show that \(y = \sum_{i=1}^{n} (y, \beta_i) \beta_i\). Since \((\alpha_i)_{i=1}^n\) is a normalized tight frame for \(\R^d\), we have  \(y = \sum_{i=1}^{n} (y, \alpha_i) \alpha_i\). On the other hand, since \(y \in S\), we have \(y = P_S(y)\) and \( (y, \alpha_i) = (y, \beta_i)\) for all \(i\). Therefore,
	\[
	y = P_S(y) = \sum_{i=1}^{n} (y, \alpha_i) P_S(\alpha_i) = \sum_{i=1}^{n} (y, \beta_i) \beta_i \,,
	\]
	as we wanted to show.
\end{proof}

\begin{remark}
	Even when the vectors \(\alpha_i\) are non-zero and pairwise distinct, their projections \(\beta_i = P_S(\alpha_i)\) may vanish or coincide. The definition of a tight frame (Definition~\ref{def:tightframe}) is intentionally permissive and places no restrictions on the vectors \((\alpha_i)_{i=1}^n\) so that Lemma~\ref{lem:projS} applies without further assumptions.
\end{remark}

\begin{example}\label{ex:naimark}
	Let \(S \subset \R^n\) be a vector subspace of dimension \(d\). Denote by \(P_S\) the orthogonal projection onto \(S\). Let \((e_i)_{i=1}^n\) be the standard basis of \(\R^n\).
	Then \((\alpha_i)_{i=1}^n\) with \(\alpha_i = P_S(e_i)\) is a normalized tight frame for \(S \cong \R^d\).
\end{example}

\begin{remark}
	Every normalized tight frame for \(\R^d\) is of the form given by Example~\ref{ex:naimark}. This is a particular instance of Na\u{\i}mark's theorem \cite[Theorem 2.2 and Exercise 2.26]{waldron}.
\end{remark}

\subsection{Trace formula and the stability inequalities}

\begin{lemma}\label{lem:trxox}
    For \(x \in \R^d\) we have
    \begin{equation}\label{eq:trxox}
        \tr (x \otimes x) = |x|^2 \,.
    \end{equation}
\end{lemma}
\begin{proof}
    The diagonal elements of the matrix \((x x^\top)_{ij} = x_i x_j\) are the squares of the components of \(x\), adding them up gives Equation~\eqref{eq:trxox}.
\end{proof}

The next result is known as the \emph{trace formula} \cite[Proposition 2.1]{waldron}.

\begin{lemma}\label{lem:traceformula}
	Let \((\alpha_i)_{i=1}^n\) be a normalized tight frame for \(\R^d\). Then
	\begin{equation}\label{eq:traceformula}
	\sum_{i=1}^{n} |\alpha_i|^2 = d \,.
	\end{equation}
\end{lemma}

\begin{proof}
    Taking the trace in Equation~\eqref{eq:tightframe} with \(\lambda=1\) we have
    \[
    \sum_{i=1}^n \tr (\alpha_i \otimes \alpha_i) =  d \,.
    \]
    By Lemma~\ref{lem:trxox}, \(\tr (\alpha_i \otimes \alpha_i) = |\alpha_i|^2\),
	and Equation~\eqref{eq:traceformula} follows.
\end{proof}

The next result bounds the sum of the squared lengths of the vectors of a normalized tight frame contained in a given subspace in terms of its dimension. 

\begin{proposition}\label{prop:stabineq}
    Suppose that \((\alpha_i)_{i=1}^n\) is a normalized tight frame for \(\R^d\).\vspace{1.5mm} \\
	\textup{(i)} 
    If  \(S \subset \R^d\) is a vector subspace then the following inequality holds:
	\begin{equation}\label{eq:stabineq}
	\sum_{i | \alpha_i \in S} |\alpha_i|^2  \leq \dim S \,.	
	\end{equation}
	\textup{(ii)} Equality holds in~\eqref{eq:stabineq} if and only if each \(\alpha_i\) lies either in \(S\) or in \(S^{\perp}\).
\end{proposition}

\begin{proof}
	(i) By Lemma~\ref{lem:projS}, the sequence \((\beta_i)_{i=1}^n\) with \(\beta_i = P_S(\alpha_i)\) is a normalized tight frame for \(S\). By Lemma~\ref{lem:traceformula} we have \(\sum_{i=1}^{n}|\beta_i|^2 = \dim S\).
	Since \(\alpha_i = \beta_i\) for all \(i\) such that \(\alpha_i \in S\), we have
	\[
	\sum_{i | \alpha_i \in S} |\alpha_i|^2  = \sum_{i | \alpha_i \in S} |\beta_i|^2 \leq \sum_{i=1}^{n} |\beta_i|^2 = \dim S \,.
	\]
	Moreover, equality holds if and only if \(\beta_i = 0\) (equivalently  \(\alpha_i \in S^{\perp}\)) for all \(i\) such that \(\alpha_i \notin S\), thus proving (ii).
\end{proof}

We refer to~\eqref{eq:stabineq} as the \emph{stability inequalities}, c.f. \cite[Theorem 11.2]{dolgachev}.

\section{Barthe's Theorem}\label{sec:barthe}

This section recalls the form of Barthe's theorem \cite{barthe} that we need in the construction of Euclidean \(\vee\)-systems
from real PK arrangements (Theorem~\ref{thm:fromPKtoV}).
We follow the presentation of this theorem given in \cite{aks}.
We begin with vector configurations, irreducibility,
and the matroid base polytope. Barthe's theorem is then used to pass from a weight vector in the relative interior of the base polytope to a normalized tight frame
in the same projective class with prescribed squared lengths. We also recall the corresponding uniqueness statement, which is used to show uniqueness up to orthogonal projective equivalence of the \(\vee\)-system associated to a real PK arrangement. 

\subsection{Vector configurations}

\begin{definition}\label{def:vc}
	A \emph{vector configuration} in \(\R^d\) is a finite set \(\mA = \{\alpha_i\}_{i=1}^n\) of vectors \(\alpha_i \in \R^d\). The vector configuration \(\mA\) is \emph{simple} if the vectors \(\alpha_i\) are non-zero and pairwise non-proportional. The vector configuration \(\mA\) is \emph{spanning} if the vectors in \(\mA\) span \(\R^d\).
\end{definition}

\begin{center}
	\emph{Throughout this paper, all our vector configurations are assumed to be simple and spanning. (In particular, not every tight frame corresponds to such a vector configuration.)}
\end{center}

\begin{definition}\label{def:irrvc}
	Let \(\mA\) be a vector configuration in \(\R^d\). We say that \(\mA\) is \emph{reducible} if there is a
	non-trivial direct sum decomposition \(\R^d = P \oplus Q\) such that \(\mA = \left(\mA \cap P \right) \cup \left(\mA \cap Q\right)\). We say that \(\mA\) is \emph{irreducible} if it is not reducible.
\end{definition}

\begin{remark}
	Let \(\mA = \{\alpha_i\}_{i=1}^n\) be a vector configuration in \(\R^d\). Let \(M = M[\mA]\) be the \emph{vector matroid} associated to \(\mA\) whose ground set is \([n] = \{1, \ldots, n\}\) and the independent subsets are given by \(I \subset [n]\) such that the vectors \(\{\alpha_i \mid i \in I\}\) are linearly independent. Then the vector configuration \(\mA\) is irreducible if and only if the matroid \(M\) is \emph{connected}, see Exercise 7 in \cite[Chapter 4.1]{oxley}. In relation with the $\vee$-systems the matroidal approach was used in \cite{lechtenfeld_et_al}, \cite{schreiberveselov}.
\end{remark}

\begin{lemma}\label{lem:rin01}
	Let \(\mA = \{\alpha_i\}_{i=1}^n\) be an irreducible vector configuration in \(\R^d\) with \(d \geq 2\). Moreover, suppose that \(\mA\) is also a normalized tight frame for \(\R^d\). Then \(|\alpha_i|^2 \in (0,1)\) for all \(\alpha_i \in \mA\).
\end{lemma}

\begin{proof}
	Take \(\alpha_i \in \mA\).
	Let \(S \subset \R^d\) be the \(1\)-dimensional vector subspace spanned by \(\alpha_i\). Proposition~\ref{prop:stabineq} (i) applied to \(S\) implies that \(|\alpha_i|^2 \leq 1\). Since \(\mA\) is irreducible and \(S\) is a proper subspace of \(\R^d\) (because \(d \geq 2\)), Proposition~\ref{prop:stabineq} (ii) implies that \(|\alpha_i|^2 < 1\).
\end{proof}

\subsection{The base polytope of a vector configuration}

\subsectionstart

\noindent 
Let \(\mA = \{\alpha_1, \ldots, \alpha_n\}\) be a vector configuration in \(\R^d\).
For a subset \(I \subset [n]\), we let \(\mathbf{1}_I\) be the \(0/1\) vector in \(\R^n\) given by the indicator function of \(I\),
\[
\mathbf{1}_I(i) = \begin{cases}
1 &\textup{if } i \in I \,, \\
0 &\textup{if } i \notin I \,.
\end{cases}
\]

\begin{definition}\label{def:basepolytope}
	The \emph{base polytope} of \(\mA\) is the convex hull \(P(\mA) \subset \R^n\) of the indicator functions \(\mathbf{1}_I\) where \(I\) ranges over all subsets \(I \subset [n]\) with \(|I|=d\) such that the vectors \(\{\alpha_i \mid i \in I\}\) form a basis of \(\R^d\).
\end{definition}

The following dual description of the base polytope is well-known.

\begin{lemma}[{\cite[Corollary 40.2d]{schrijver}}]\label{lem:basepolytope-ineq}
	The base polytope \(P(\mA)\) of \(\mA \subset \R^d\) consists of all points \(w = (w_1, \ldots, w_n)\) in \((\R_{\geq 0})^n\) such that
	\begin{equation}\label{eq:basepolytope-ineq}
	\sum_{i=1}^n w_i = d \quad\textup{and} \quad \sum_{i \mid \alpha_i \in S} w_i \leq \dim S
	\end{equation}
	for every linear subspace \(S \subset \R^d\) spanned by vectors in \(\mA\).
\end{lemma}

It follows immediately from Definition~\ref{def:basepolytope} that the base polytope \(P(\mA)\) is contained in the affine hyperplane \(\{\sum_i w_i = d \} \subset \R^n\), thus \(\dim P(\mA) \leq n-1\).
The next result characterizes the vector configurations for which \(\dim P(\mA) = n-1\).

\begin{lemma}[{\cite[Theorem 1.12.9]{borovikgelfandneil}}]\label{lem:basepolytope-dim}
	The base polytope has maximal possible dimension \(\dim P(\mA) = n-1\) if and only if the vector configuration \(\mA\) is irreducible.
\end{lemma}

\subsection{Barthe's theorem}

\begin{theorem}[Barthe \cite{barthe}]\label{thm:barthe}
	Let \(\mA = \{\alpha_i\}_{i=1}^n\) be a vector configuration in \(\R^d\) and let \(w = (w_1, \ldots, w_n)\) be a weight vector in \(\R^n\) such that \(w_i > 0\) for all \(i\) and \(\sum_i w_i = d\). Then the following are equivalent.
	\begin{enumerate}[leftmargin=*, label=\textup{(\arabic*)}]
		\item There exist a linear transformation \(F \in \GL(d)\) and non-zero scalars \(\lambda_i \in \R\) such that the vector configuration \(\mB = \{\beta_i\}_{i=1}^n\) with \(\beta_i = \lambda_i \, F(\alpha_i)\) satisfies:
		\begin{enumerate}[label=\textup{(\roman*)}]
			\item \(\mB\) is a normalized tight frame for \(\R^d\);
			\item \(|\beta_i|^2 = w_i\) for all \(i\).
		\end{enumerate}
		\item The weight vector \(w\) belongs to the relative interior of the base polytope \(P(\mA)\), i.e., \(w \in \relin(P(\mA))\), where \(\relin(P(\mA))\) is the interior of \(P(\mA)\) in the minimal affine subspace of \(\R^n\) that contains it.
	\end{enumerate}
\end{theorem}

\begin{proof}
	We follow the presentation in \cite{aks}. The vector configuration \(\mA = \{\alpha_i\}_{i=1}^n\) can be put into \emph{radial \(w\)-isotropic position} \cite[Definitions 1.2 and 1.3]{aks} if there is an invertible linear transformation \(F \in \GL(d)\) such that
	\begin{equation}\label{eq:rip}
	\sum_{i=1}^n w_i \, \left( \frac{F(\alpha_i)}{|F(\alpha_i)|} \otimes \frac{F(\alpha_i)}{|F(\alpha_i)|}  \right) = \Id \,.
	\end{equation}
	In this case, the vector configuration \(\mB = \{\beta_i\}_{i=1}^n\) with
	\[
	\beta_i = \frac{\sqrt{w_i}}{|F(\alpha_i)|} \cdot F(\alpha_i) 
	\]
	satisfies (i) and (ii) in condition (1) of the theorem:
	\begin{equation}\label{eq:betas}
		\sum_{i=1}^{n} \beta_i \otimes \beta_i = \Id \quad \textup{and} \quad |\beta_i|^2 = w_i \,.
	\end{equation}
	Conversely, if there is \(F \in \GL(d)\) and non-zero \(\lambda_i \in \R\) such that the vectors \(\beta_i = \lambda_i \, F(\alpha_i)\) satisfy Equation~\eqref{eq:betas}, then
	\[
	\sum_{i=1}^n w_i \, \left( \frac{F(\alpha_i)}{|F(\alpha_i)|} \otimes \frac{F(\alpha_i)}{|F(\alpha_i)|}  \right) = \sum_{i=1}^{n} w_i \, \left( \frac{\beta_i}{|\beta_i|} \otimes \frac{\beta_i}{|\beta_i|} \right) = \Id \,.
	\]
	This shows that item (1) of the theorem is equivalent to the statement: 
	\begin{center}
		(1)' \(\mA\) can be put into radial \(w\)-isotropic position.
	\end{center}
	The equivalence (1)' \(\iff\) (2) then follows from \cite[Theorem A.1]{aks}.
\end{proof}

\subsection{Uniqueness in Barthe's theorem}

\begin{definition}
	Let \(\mA = \{\alpha_i\}_{i=1}^n\) and  \(\mB = \{\beta_i\}_{i=1}^n\) be vector configurations in \(\R^d\). We say that \(\mA\) and \(\mB\) are \emph{projectively equivalent} if there exist an invertible linear transformation \(F \in \GL(d)\) and non-zero scalars \(\lambda_i \in \R\) such that \(\beta_i = \lambda_i \, F(\alpha_i)\) for all \(i\). Equivalently, \(\mA\) and \(\mB\) are \emph{projectively equivalent} if there exists \(F \in \GL(d)\) such that \(M_i = F(L_i)\) for all \(i\), where \(M_i\) and \(L_i\) are the \(1\)-dimensional subspaces \(M_i = \R \cdot \beta_i\) and \(L_i = \R \cdot \alpha_i\).
\end{definition}

\begin{remark}
	If \(\mA\) and \(\mB\) are projectively equivalent then their base polytopes are equal.
\end{remark}

\begin{definition}\label{def:orthogonally-proj-equiv}
	Let \(\mA = \{\alpha_i\}_{i=1}^n\) and  \(\mB = \{\beta_i\}_{i=1}^n\) be vector configurations in \(\R^d\). 
	We say that \(\mA\) and \(\mB\) are \emph{orthogonal projectively equivalent} if there exists a linear isometry \(G \in \mathrm{O}(d)\) such that \(\beta_i = \pm \, G(\alpha_i)\) for all \(i\).
\end{definition}

The next result is a uniqueness statement, up to orthogonal projective equivalence, for the tight frames in Barthe's Theorem~\ref{thm:barthe}. 

\begin{proposition}[{\cite[Lemma A.7]{aks}}]  \label{prop:uniqueness}
	Let \(\mV = \{v_i\}_{i=1}^n\) be an irreducible vector configuration in \(\R^d\).
	Suppose that \(\mA = \{\alpha_i\}_{i=1}^n\) and \(\mB = \{\beta_i\}_{i=1}^n\) are vector configurations projectively equivalent to \(\mV\) such that
	\begin{equation}\label{eq:betas2}
	\sum_{i=1}^{n} \alpha_i \otimes \alpha_i = \sum_{i=1}^{n} \beta_i \otimes \beta_i = \Id \quad \textup{and} \quad |\alpha_i|^2 = |\beta_i|^2 \quad \textup{for all } i \,.
	\end{equation}
	Then \(\mA\) and \(\mB\) are orthogonal projectively equivalent.
\end{proposition}

\begin{proof}
	The assumption that \(\mA\) and \(\mB\) are projectively equivalent implies that there is \(F \in \GL(d)\) mapping the lines \(\R \cdot \alpha_i\) to the lines \(\R \cdot \beta_i\) for all \(i\).
	For each \(1 \leq i \leq n\) choose a unit vector \(x_i\) in the line \(\R \cdot \alpha_i\) and let \(y_i = F(x_i)/|F(x_i)|\). Then
	\[
	\alpha_i \otimes \alpha_i = |\alpha_i|^2 \cdot   (x_i \otimes x_i) \quad \textup{and} \quad  \beta_i \otimes \beta_i = |\beta_i|^2 \cdot (y_i \otimes y_i) \,.
	\]
	Using Equation~\eqref{eq:betas2}, we obtain that
	\[
	\sum_{i=1}^{n} w_i \cdot (x_i \otimes x_i) = \sum_{i=1}^{n} w_i \cdot (y_i \otimes y_i) = \Id \,,
	\]
	where \(w_i = |\alpha_i|^2 = |\beta_i|^2\). It then follows from \cite[Lemma A.7]{aks} that there is a linear isometry \(G \in \mathrm{O}(d)\) and \(\lambda>0\) such that \(F = \lambda \cdot G\) and \(y_i = G(x_i)\) for all \(i\). Since \(\alpha_i = \pm \sqrt{w_i} \cdot x_i\) and \(\beta_i = \pm \sqrt{w_i} \cdot y_i\), it follows that \(\beta_i = \pm G(\alpha_i)\) for all \(i\), as wanted.
\end{proof}

\subsection{Moment maps}

\subsectionstart

\noindent
We interpret the tight frame condition as a balancing condition for a weighted configuration of points in projective space. Although this is not used elsewhere in the paper, it places our work within the broader framework of \emph{moment maps} in differential geometry.

Let \(\Phi\) be the map from the unit sphere \(S^{d-1} \subset \R^d\) to the vector space \(\Sym_0(\R^d)\) of trace-free symmetric \(d \times d\) matrices defined by
\begin{equation*}
\Phi(x) =  x \otimes x - \frac{1}d \Id \,.
\end{equation*}
Since \(\Phi(x)=\Phi(-x)\), the map \(\Phi\) defines an embedding of \(\RP^{d-1}\) into \(\Sym_0(\R^d)\). 
Consider now a vector configuration \(\mA = \{\alpha_i\}_{i=1}^n\) in \(\R^d\).
It is easy to check that
\begin{equation}\label{eq:zcm}
	\sum_{i=1}^{n} \alpha_i \otimes \alpha_i = \lambda \cdot \Id \, \textup{ for some } \lambda > 0 \iff
	\sum_{i=1}^n w_i \cdot \Phi(\hat{\alpha}_i)  = 0 \,,
\end{equation}
where \(w_i = |\alpha_i|^2\) and \(\hat{\alpha}_i = \alpha_i / |\alpha_i|\).
Thus the vector configuration \(\mA\) is a tight frame for \(\R^d\) if and only if the centre of mass of the points \(\Phi(\hat{\alpha}_i) \in \Sym_0(\R^d)\) with weights \(w_i\) is equal to \(0\). 	
Furthermore, if the weights \(w_i\) are integers, then~\eqref{eq:zcm} can be interpreted as a zero moment map equation, and Barthe’s Theorem can be recovered from the Kempf-–Ness Theorem \cite[Theorem 5.19]{szek} in Geometric Invariant Theory.

\section{The Hirzebruch quadratic form and the V-conditions}\label{sec:hqf}

This section relates the Hirzebruch quadratic form to the Euclidean \(\vee\)-conditions. We first recall the frame-potential inequality for vector sequences (Proposition~\ref{prop:tfineq}). We then rewrite the Euclidean \(\vee\)-conditions in terms of tight frames on the ambient space and on the
two-dimensional subspaces spanned by the configuration (Lemma~\ref{lem:vstf}). Finally, following the variational perspective of \cite{dbp-hirzebruch}, we compare the \(\vee\)-conditions
with the Hirzebruch quadratic form and obtain Proposition~\ref{prop:hirineq} which is the key to
the proofs of the main theorems.

\subsection{The frame-potential inequality}

\subsectionstart

\noindent Let \((\alpha_i)_{i=1}^n\) be a finite sequence of vectors in \(\R^d\).
Let \(A\) be the symmetric matrix defined by
\begin{equation}
	A = \sum_{i=1}^n \alpha_i \, \alpha_i^{\top}
\end{equation}
also known as the \emph{frame operator} \cite[Chapter 2.4]{waldron}. 
Consider the decomposition of the matrix \(A\) into scalar and trace-free parts given by
\begin{equation}\label{eq:stfdec}
	A = \frac{\tr(A)}{d} \Id + \left( A - \frac{\tr(A)}{d} \Id \right) \,.
\end{equation}
Write \(\lambda = \tr(A) /d\) and \(A^{\circ} = A - \lambda \Id\), so the above decomposition is \(A = \lambda \Id + A^{\circ}\). 

We equip the vector space of symmetric matrices \(\Sym(\R^d)\) with its standard positive definite inner product
\[
(X,Y)=\tr(XY)\,, \quad |X|=\sqrt{\tr(X^2)}\,.
\]
Then the decomposition~\eqref{eq:stfdec} is orthogonal with respect to \(\inn\).

\begin{lemma}\label{lem:tracefreenorm}
	The following identity holds:
	\begin{equation}\label{eq:tracefreenorm}
	|A^{\circ}|^2 = \sum_{i=1}^{n} \sum_{j=1}^n (\alpha_i, \alpha_j)^2 - \frac{1}{d} \cdot \left( \sum_{i=1}^{n} |\alpha_i|^2 \right)^2 
	\end{equation}
\end{lemma}

\begin{proof}
	Since the decomposition~\eqref{eq:stfdec} is orthogonal,
	\begin{equation}\label{eq:tfpf1}
		|A|^2 =  |\lambda \Id|^2 + |A^{\circ}|^2 \,.
	\end{equation}
	Clearly, \(|\Id|^2 = d\) and \(|\lambda \Id|^2 = \tr(A)^2/d\). 
	By Lemma~\ref{lem:trxox}, \(\tr(A) = \sum_i |\alpha_i|^2\), so
	\begin{equation}\label{eq:tfpf2}
		|\lambda \Id|^2 = \frac{1}{d} \cdot \left( \sum_{i=1}^{n} |\alpha_i|^2 \right)^2 \,.	
	\end{equation}
	On the other hand, 
	\[
	A^2 = \sum_{i=1}^{n} \sum_{j=1}^n \alpha_i \alpha_i^{\top}\alpha_j \alpha_j^{\top} \,.
	\] 
	It is easy to see that \(\tr\big(\alpha_i \alpha_i^{\top}\alpha_j \alpha_j^{\top} \big) = (\alpha_i, \alpha_j)^2\). Thus
	\begin{equation}\label{eq:tfpf3}
	|A|^2 =\tr(A^2)= \sum_{i=1}^{n} \sum_{j=1}^n (\alpha_i, \alpha_j)^2 \,.
	\end{equation}
	Equation~\eqref{eq:tracefreenorm} follows from Equations~\eqref{eq:tfpf1},~\eqref{eq:tfpf2}, and~\eqref{eq:tfpf3}.
\end{proof}

The following well-known variational characterization of tight frames plays a crucial role in our arguments.

\begin{proposition}[{\cite[Theorem 6.1]{waldron}}] \label{prop:tfineq}
	Let \((\alpha_i)_{i=1}^n\) be a finite sequence of vectors in \(\R^d\).
	Then the following inequality holds:
	\begin{equation}\label{eq:tfineq}
	\sum_{i=1}^{n} \sum_{j=1}^n (\alpha_i, \alpha_j)^2 \geq \frac{1}{d} \cdot \left( \sum_{i=1}^{n} |\alpha_i|^2 \right)^2 \,.
	\end{equation}
	Moreover, if at least one of the vectors \(\alpha_i\) is non-zero, then the equality holds if and only if \((\alpha_i)_{i=1}^n\) is a tight frame for \(\R^d\).
\end{proposition}

\begin{proof}
	The inequality follows immediately from Equation~\eqref{eq:tracefreenorm}.
	Moreover, equality holds if and only if \(A = \lambda \, \Id\). The assumption that at least one vector \(\alpha_i\) is non-zero implies that \(\lambda>0\). Thus the equality holds if and only if the frame operator is a positive multiple of the identity, i.e., \((\alpha_i)_{i=1}^n\) is a tight frame for \(\R^d\).
\end{proof}

\begin{remark}
	If \(|\alpha_i|=1\) for all \(i\) the inequality~\eqref{eq:tfineq} is a particular case of a family of inequalities known as the \emph{Welch bounds}.
\end{remark}
\vspace{-2mm}

\subsection{Euclidean \(\vee\)-systems as tight frames}

\begin{lemma}\label{lem:vstf} 
	Let \(\mA = \{\alpha_i\}_{i=1}^n\) be a vector configuration in \(\R^d\). Then \(\mA\) is a Euclidean \(\vee\)-system if and only if the following conditions are satisfied.
	\begin{enumerate}[label=\textup{(V\arabic*)'}]
		\item \label{it:V1'} \(\mA\) is a normalized tight frame for \(\R^d\).
		
		\item \label{it:V2'} For every linear \(2\)-plane \(\Pi \subset \R^d\) with \(|\mA \cap \Pi| \geq 3\),  \(\mA \cap \Pi\) is a tight frame for \(\Pi\).
		
		\item \label{it:V3'} For every linear \(2\)-plane \(\Pi \subset \R^d\) with \(\mA \cap \Pi = \{\alpha_i, \alpha_j\}\) for some \(i \neq j\),  \((\alpha_i, \alpha_j) = 0\).
	\end{enumerate}
\end{lemma}

\begin{proof}
	It follows from Lemma~\ref{lem:tfequiv} that items~\ref{it:V1} and~\ref{it:V2} of Definition~\ref{def:Vsystem} are equivalent to~\ref{it:V1'} and~\ref{it:V2'} respectively. Items~\ref{it:V3} and~\ref{it:V3'} are simply the same.
\end{proof}
\vspace{-2mm}

\subsection{The associated weighted arrangement}

\begin{definition}\label{def:associatedwa}
	Let \(\mA = \{\alpha_i\}_{i=1}^n\) be a vector configuration in \(\R^d\).
	The \emph{associated weighted arrangement} is the pair \((\Delta_{\mA}, r_\mA)\) consisting of hyperplanes \(\Delta_{\mA} = \{\Delta_1, \ldots, \Delta_n\}\) and weights \(r_{\mA} = (r_1, \ldots, r_n)\) with  
	\begin{equation*}
	\Delta_i = \alpha_i^{\perp} \quad \textup{and} \quad
	r_i = |\alpha_i|^2 \,.
	\end{equation*}
	
	We also say that \(\Delta_{\mA}\) is the \emph{associated hyperplane arrangement} and that \(r_{\mA}\) is the \emph{associated weight vector}.
\end{definition}

\begin{lemma}\label{lem:associated-arrangement-essential}
Let \(\mA=\{\alpha_i\}_{i=1}^n\) be a Euclidean \(\vee\)-system in \(\mathbb R^d\). Then the associated hyperplane arrangement \(\Delta_{\mA}\) is essential, i.e. \(\bigcap_{i=1}^n \alpha_i^\perp=\{0\}\).
\end{lemma}

\begin{proof}
By Lemma~\ref{lem:vstf}, \(\mA\) is a normalized tight frame for \(\mathbb R^d\). Hence the vectors \(\alpha_1,\ldots,\alpha_n\) span \(\mathbb R^d\) by Lemma~\ref{lem:span}. Therefore the only vector orthogonal to all \(\alpha_i\) is \(0\).
\end{proof}

\begin{definition}\label{def:linear-equivalence}
	Let \(\mH = \{H_i\}_{i=1}^n\) and \(\mH' = \{H'_i\}_{i=1}^n\) be two arrangements of linear hyperplanes in \(\R^d\). We say that \(\mH\) and \(\mH'\) are \emph{linearly equivalent} if there is an invertible linear transformation \(F \in \GL(d)\) such that \(H_i' = F(H_i)\) for all \(i\). Two weighted arrangements are linearly equivalent if their corresponding hyperplane arrangements are.
\end{definition}

\begin{lemma}\label{lem:lehqf}
	Let \(\mH = \{H_i\}_{i=1}^n\) and \(\mH' = \{H'_i\}_{i=1}^n\) be two arrangements of linear hyperplanes in \(\R^d\). Let \(Q\) and \(Q'\) be their respective Hirzebruch quadratic forms. Suppose that \(\mH\) and \(\mH'\) are linearly equivalent. Then \(Q'=Q\).
\end{lemma}

\begin{proof}
	Immediate from Definition~\ref{def:hirqf}.
\end{proof}

\begin{lemma}\label{lem:linequivha}
	Let \(\mA = \{\alpha_i\}_{i=1}^n\) and \(\mB = \{\beta_i\}_{i=1}^n\) be vector configurations in \(\R^d\) with associated hyperplane arrangements \(\Delta_{\mA}\) and \(\Delta_{\mB}\) respectively. Then \(\Delta_{\mA}\) is linearly equivalent to \(\Delta_{\mB}\) if and only if \(\mA\) is projectively equivalent to \(\mB\).
\end{lemma}

\begin{proof}  A vector configuration \(\mA\in \mathbb R^d\)  defines the dual hyperplane arrangement in \({\mathbb R^d}^*\), which is isomorphic to  \(\Delta_{\mA}\).  Clearly \(\mA\) and \(\mB\) are projectively isomorphic if and only if their dual arrangements are linearly isomorphic.
\end{proof}
\vspace{-1mm}

\subsection{Variational characterization of Euclidean \(\vee\)-systems}

\subsectionstart

\noindent Let \(\mA = \{\alpha_i\}_{i=1}^n\) be a vector configuration in \(\R^d\) and let \(Q: \R^n \to \R\) be the Hirzebruch quadratic form of the associated hyperplane arrangement \(\Delta_{\mA}\).

\begin{notation} Denote by  \(R\) be the subset of ordered pairs of indices \((i,j) \in [n] \times [n]\) such that 
	\[
	| \Span \, \{\alpha_i, \alpha_j\} \cap \mA | = 2. 
	\]
	Denote by \(\mF\)  the set of linear \(2\)-planes \(\Pi \subset \R^d\) with \(|\mA \cap \Pi| \geq 3\).
\end{notation} 

\begin{lemma}
	Suppose that \(\mA\) is a tight frame for \(\R^d\).
	Then the following identity holds:
	\begin{equation}\label{eq:hirzebruchnorm}
	-2\,Q(|\alpha_1|^2, \ldots, |\alpha_n|^2) = \sum_{(i,j) \in R} (\alpha_i, \alpha_j)^2 + \sum_{\Pi \in \mF} |A_{\Pi} - \nu(\Pi) \Id_{\Pi}|^2 \,,
	\end{equation}
	where \(A_{\Pi}\) for \(\Pi \in \mF\) is the frame operator of \(\mA \cap \Pi\), i.e., the linear endomorphism of \(\Pi\) given by
	\begin{equation*}
		A_{\Pi} = \sum_{i \,|\, \alpha_i \in \Pi} \alpha_i \otimes \alpha_i \,;
	\end{equation*} 
	and 
	\[
	\nu(\Pi) = \frac{\tr(A_{\Pi})}{2} = \frac{1}{2} \cdot \sum_{i \mid \alpha_i \in \Pi} |\alpha_i|^2 \,.
	\]
\end{lemma}

\begin{proof}
	We follow the proof of \cite[Theorem 3.1]{dbp-hirzebruch}.
	Since  \(\mA\) is a tight frame, by Proposition~\ref{prop:tfineq} we have
	\begin{equation}\label{eq:pfineq1}
	\frac{1}{d} \cdot \left( \sum_{i=1}^{n} |\alpha_i|^2 \right)^2  =
	\sum_{i=1}^{n} \sum_{j=1}^{n} (\alpha_i, \alpha_j)^2 \,.    
	\end{equation}
	The standing assumption that the vectors in \(\mA\) are pairwise non-proportional implies that if \(i \neq j\) then \((i, j) \in R\) or there exists a unique \(\Pi \in \mF\) such that \(\alpha_i, \alpha_j \in \Pi\). Therefore,
	\begin{equation}\label{eq:pfineq2}
	\sum_{i=1}^{n} \sum_{j=1}^{n}
	(\alpha_i, \alpha_j)^2 = \sum_{(i,j) \in R} (\alpha_i, \alpha_j)^2 \,+ \, \sum_{\Pi \in \mF} \sum_{\,\,\alpha_i, \alpha_j \in \Pi} (\alpha_i, \alpha_j)^2 \,-\, \sum_{i=1}^{n} (\sigma_i-1) \cdot |\alpha_i|^4 \,,    
	\end{equation}
	where \(\sigma_i\) is the number of \(2\)-planes in \(\mF\) that contain the vector \(\alpha_i\), equivalently the number \(\sigma_i\) from Definition~\ref{def:hirqf} for the associated arrangement \(\Delta_{\mA}\).
	Lemma~\ref{lem:tracefreenorm} applied to the vectors \(\mA \cap \Pi\) for a \(2\)-plane \(\Pi \in \mF\) gives us
	\begin{equation}\label{eq:pfineq3}
	\sum_{\alpha_i, \alpha_j \in \Pi} (\alpha_i, \alpha_j)^2 = |A_{\Pi} - \nu(\Pi) \Id_{\Pi}|^2 + 2 \, \nu(\Pi)^2 \,.
	\end{equation}
	It follows from~\eqref{eq:pfineq1},~\eqref{eq:pfineq2}, and~\eqref{eq:pfineq3}  that
	\[
	- 2 \cdot \sum_{\Pi \in \mF} \nu(\Pi)^2 \,+\, \sum_{i=1}^{n} (\sigma_i-1) \cdot |\alpha_i|^4 + \frac{1}{d} \cdot \left( \sum_{i=1}^{n} |\alpha_i|^2 \right)^2  = \sum_{(i,j) \in R} (\alpha_i, \alpha_j)^2 + \sum_{\Pi \in \mF} |A_{\Pi} - \nu(\Pi) \Id_{\Pi}|^2  \,.
	\]
	The left hand side of the above equality is equal to \(-2Q(|\alpha_1|^2, \ldots, |\alpha_n|^2)\) by direct comparison with Definition~\ref{def:hirqf}. This proves Equation~\eqref{eq:hirzebruchnorm}.
\end{proof}

\begin{proposition}\label{prop:hirineq}
	Suppose that \(\mA\) is normalized tight frame. Then 
	\begin{equation}\label{eq:hirineq}
	   Q(|\alpha_1|^2, \ldots, |\alpha_n|^2)  \leq 0 \,.
	\end{equation}
     Moreover, equality holds if and only if \(\mA\) is a Euclidean \(\vee\)-system.
\end{proposition}

\begin{proof}
	The inequality is a direct consequence of Equation~\eqref{eq:hirzebruchnorm}. Equality holds if and only if
	(i) \((\alpha_i, \alpha_j) = 0\) for all \((i,j) \in R\); and (ii) \(A_{\Pi} = \nu(\Pi) \Id_{\Pi}\) (or equivalently \(\mA \cap \Pi\) is a tight frame for \(\Pi\)) for every \(2\)-plane \(\Pi \in \mF\). The result then follows from the characterization of Euclidean \(\vee\)-systems as tight frames given by Lemma~\ref{lem:vstf}.  
\end{proof}

\section{From \(\vee\)-systems to PK arrangements, Theorem~\ref{thm:fromVtoPK}}\label{sec:vtop}

We prove in this section that an irreducible Euclidean \(\vee\)-system determines a real PK arrangement. The proof separates the three conditions in the Definition~\ref{def:PK} of a real PK arrangement. The trace formula gives the total weight condition~\ref{it:PK1}, the stability
inequality for tight frames gives the strict inequalities along intersections of
hyperplanes~\ref{it:PK2}, and the variational characterization from Proposition~\ref{prop:hirineq} gives the vanishing of the Hirzebruch quadratic form~\ref{it:PK3}.

\subsection{Proof of Theorem~\ref{thm:fromVtoPK}}
Let \(\mA = \{\alpha_i\}_{i=1}^n\) be an irreducible Euclidean \(\vee\)-system in \(\R^d\), with \(d \geq 2\).
Let \((\Delta_{\mA}, r_{\mA})\), where \(\Delta_{\mA} = \{\Delta_1, \ldots, \Delta_n\}\) and \(r_{\mA} = (r_1, \ldots, r_n)\), denote the associated weighted arrangement as in Definition~\ref{def:associatedwa}. We prove that \((\Delta_{\mA}, r_{\mA})\) satisfies items~\ref{it:PK1},~\ref{it:PK2}, and~\ref{it:PK3} of Definition~\ref{def:PK}, and hence defines a real PK arrangement.

Since \(\mA\) is a Euclidean \(\vee\)-system, it follows in particular from Lemma~\ref{lem:vstf} that \(\mA\) is a normalized tight frame.
The next Lemmas~\ref{lem:PK0},~\ref{lem:P1}, and~\ref{lem:P2} rely only on the assumptions that \(\mA\) is an irreducible and normalized tight frame for \(\R^d\), with \(d \geq 2\). 

\begin{lemma}\label{lem:PK0}
	The weights \(r_i = |\alpha_i|^2\) satisfy \(r_i \in (0,1)\) for all \(i\).
\end{lemma}

\begin{proof}
	This follows from Lemma~\ref{lem:rin01}.
\end{proof}

\begin{lemma}\label{lem:P1}
    \((\Delta_{\mA}, r_{\mA})\) satisfies~\ref{it:PK1}.
\end{lemma}

\begin{proof}
    This follows from Lemma~\ref{lem:traceformula}.
\end{proof}

\begin{lemma}\label{lem:P2}
    \((\Delta_{\mA}, r_{\mA})\) satisfies~\ref{it:PK2}.
\end{lemma}

\begin{proof}
    Let \(L \subset \R^d\) be a non-zero proper linear subspace obtained as intersection of hyperplanes \(\Delta_i \in \Delta_{\mA}\). We want to show that
    \begin{equation}\label{eq:stinpf}
        \sum_{i \,\mid\, L \subset \Delta_i} r_i < r(L) \,,
    \end{equation}
    where \(r(L)\) is the codimension of \(L\). 
    Let \(S \subset \R^d\) be the orthogonal complement \(S = L^{\perp}\). Then \(L \subset \Delta_i = \alpha_i^{\perp}\) if and only if \(\alpha_i \in S\). Therefore,
    \begin{equation}\label{eq:pfstin1}
      \sum_{i \,|\, L \subset \Delta_i} r_i = \sum_{i \,|\, \alpha_i \in S} |\alpha_i|^2  \,.
    \end{equation}
    By Proposition~\ref{prop:stabineq} (i), 
    \begin{equation}\label{eq:pfstin2}
        \sum_{i \,|\, \alpha_i \in S} |\alpha_i|^2 \leq \dim S = r(L) \,.
    \end{equation}
    
	By Proposition~\ref{prop:stabineq} (ii), if equality holds
    in~\eqref{eq:pfstin2}, then \(\mA = (\mA \cap S) \cup (\mA \cap S^{\perp})\) which contradicts the assumption that \(\mA\) is irreducible. Therefore, the inequality~\eqref{eq:pfstin2} must be strict. Equation~\eqref{eq:pfstin1} together with the strict inequality~\eqref{eq:pfstin2} imply~\eqref{eq:stinpf}, as wanted.
\end{proof}

\begin{lemma}\label{lem:P3}
    \((\Delta_{\mA}, r_{\mA})\) satisfies~\ref{it:PK3}.
\end{lemma}

\begin{proof}
    Let \(Q\) be the Hirzebruch quadratic form of \(\Delta_{\mA}\). We want to show that \(Q(r_1, \ldots, r_n) = 0\), where \(r_i = | \alpha_i|^2\). This follows from
    Proposition~\ref{prop:hirineq}.
\end{proof}

\begin{proof}[Proof of Theorem~\ref{thm:fromVtoPK}]
Let \(\mA = \{\alpha_i\}_{i=1}^n\) be an irreducible Euclidean \(\vee\)-system in \(\R^d\) with \(d \geq 2\). Let \((\Delta_{\mA}, r_{\mA})\) be the associated weighted arrangement. By Lemma~\ref{lem:PK0}, the weights \(r_i\) are in the open interval \((0,1)\).
By Lemmas~\ref{lem:P1},~\ref{lem:P2}, and~\ref{lem:P3} the weighted arrangement \((\Delta_{\mA}, r_{\mA})\) satisfies~\ref{it:PK1},~\ref{it:PK2}, and~\ref{it:PK3}. Hence, \((\Delta_{\mA}, r_{\mA})\) is a real PK arrangement.     
\end{proof}

\section{From PK arrangements to \(\vee\)-systems, Theorem~\ref{thm:fromPKtoV}}\label{sec:ptov}

We now prove the converse direction. Starting from a real PK arrangement, we choose normal vectors to its hyperplanes. The PK conditions~\ref{it:PK1} and~\ref{it:PK2} imply that the resulting
vector configuration is irreducible and that the given weight vector lies in the relative interior of its base polytope. Barthe's theorem then produces a
normalized tight frame in the required projective class and with the prescribed squared lengths. The vanishing of the Hirzebruch quadratic form~\ref{it:PK3} then implies, via Proposition~\ref{prop:hirineq}, that this tight frame is a Euclidean \(\vee\)-system.

\subsection{Proof of Theorem~\ref{thm:fromPKtoV}}

Let \((\mH, w) = \{(H_i, w_i)\}_{i=1}^n\) be a real PK arrangement in \(\R^d\) consisting of linear hyperplanes \(H_i \subset \R^d\) together with weights \(w_i \in (0,1)\) satisfying~\ref{it:PK1},~\ref{it:PK2}, and~\ref{it:PK3}.
For each hyperplane \(H_i \in \mH\) choose a non-zero vector \(v_i\) orthogonal to \(H_i\), so that \(H_i = \{x \in \R^d \mid (x, v_i) = 0\}\).
The vector configuration \(\mV = \{v_i\}_{i=1}^n\) is simple (because the hyperplanes \(H_i\) are pairwise distinct)
and spanning, as guaranteed by the following.

\begin{lemma}
	The set of vectors \(\mV = \{v_1, \ldots, v_n\}\) spans \(\R^d\).
\end{lemma}

\begin{proof}
	By contradiction, suppose that \(\Span \{v_1, \ldots, v_n\}\) is a proper subspace of \(\R^d\). Then there is a non-zero vector \(x \in \R^d\) such that \((x, v_i) = 0\) for all \(i\). Thus \(x \in L\), where \(L = \bigcap_{i=1}^n H_i\). In particular, \(\dim L >0\) , so \(r(L) = d - \dim L < d\). By~\ref{it:PK2},
	\[
	\sum_{i=1}^n w_i = \sum_{i \mid L \subset H_i} w_i < d \,,
	\]
	but this contradicts~\ref{it:PK1}, namely that \(\sum_{i=1}^{n}w_i = d\).
\end{proof}

Let \(P(\mV)\) be the base polytope of the vector configuration \(\mV\).

\begin{lemma}\label{lem:PK-basepolytope}
	The following holds: \textup{(i)}  \(\dim P(\mV) = n-1\); and \textup{(ii)} \(w \in \relin(P(\mV))\).
\end{lemma}

\begin{proof}
	Let \(\mS\) denote the finite collection of non-zero proper vector subspaces \(S \subset \R^d\) spanned by vectors in \(\mV\).
	Let \(C\) be the open polyhedral cone in the positive octant \(\R_{>0}^n\) defined by the linear inequalities
	\[
	\forall S \in \mS : \quad
	\frac{1}{\dim S} \cdot \sum_{i \mid v_i \in S} w_i < \frac{1}{d} \cdot \sum_{i=1}^{n} w_i \,.
	\]
	Let \(\Sigma_d \subset \R^n\) be the affine hyperplane \(\Sigma_d = \{\sum_{i=1}^{n} w_i = d\}\). 
	Items~\ref{it:PK1} and~\ref{it:PK2} together imply that \(w \in C \cap \Sigma_d\). On the other hand,
	it follows from the description of the base polytope given by Lemma~\ref{lem:basepolytope-ineq} that \(C \cap \Sigma_d \subset P(\mV)\). Since \(C \cap \Sigma_d\) is a non-empty open subset of \(\Sigma_d\) contained in \(P(\mV)\), the polytope \(P(\mV)\) must have maximal possible dimension \(\dim P(\mV) = n-1\); which proves (i). Moreover, \(w \in C \cap \Sigma_d \subset \relin(P(\mV))\), which proves (ii).
\end{proof}

\begin{lemma}\label{lem:Virr}
	The vector configuration \(\mV\) is irreducible.
\end{lemma}

\begin{proof}
	This follows from Lemma~\ref{lem:PK-basepolytope} (i) and Lemma~\ref{lem:basepolytope-dim}.
\end{proof}

\begin{lemma}\label{lem:barthe}
     There exist a linear transformation \(F \in \GL(d)\) and non-zero scalars \(\lambda_i \in \R\) such that the vector configuration \(\mA = \{\alpha_i\}_{i=1}^n\) with \(\alpha_i = \lambda_i \, F(v_i)\) satisfies:
     \begin{enumerate}[label=\textup{(\roman*)}]
     	\item \(\mA\) is a normalized tight frame for \(\R^d\);
     	\item \(|\alpha_i|^2 = w_i\) for all \(i\).
     \end{enumerate}
\end{lemma}

\begin{proof}
	This follows from Lemma~\ref{lem:PK-basepolytope} (ii) together with Barthe's Theorem~\ref{thm:barthe}.
\end{proof}

Let \(\mA\) be as in Lemma~\ref{lem:barthe} and  \(\Delta_{\mA}\) the associated hyperplane arrangement.

\begin{lemma}\label{lem:equivH}
	The arrangement \(\Delta_{\mA}\) is linearly equivalent to \(\mH\).
\end{lemma}

\begin{proof}
	By construction, the associated hyperplane arrangement of the vector configuration \(\mV\) is \(\Delta_{\mV} = \mH\). Since the vector configurations \(\mV\) and \(\mA\) are projectively equivalent (by Lemma~\ref{lem:barthe}), their associated hyperplane arrangements \(\mH\) and \(\Delta_{\mA}\) are linearly equivalent as guaranteed by Lemma~\ref{lem:linequivha} .  
\end{proof}

\begin{lemma}\label{lem:Avs}
	 The vector configuration \(\mA\) is an irreducible Euclidean \(\vee\)-system.
\end{lemma}

\begin{proof}
	Let \(Q'\) be the Hirzebruch quadratic form of \(\Delta_{\mA}\). By Lemmas~\ref{lem:equivH} and~\ref{lem:lehqf}, \(Q' = Q\), where \(Q\) is the Hirzebruch quadratic form of \(\mH\). On the other hand, by condition~\ref{it:PK3} and Lemma~\ref{lem:barthe} (ii),
	\begin{equation}\label{eq:zerohir}
		Q(|\alpha_1|^2, \ldots, |\alpha_n|^2) = Q(w_1, \ldots, w_n) = 0 \,.
	\end{equation}
	
    By Lemma~\ref{lem:barthe} (i), \(\mA\) is a normalized tight frame for \(\R^d\). It follows from Proposition~\ref{prop:hirineq} together with Equation~\eqref{eq:zerohir} that \(\mA\) is a Euclidean \(\vee\)-system. 
    
    The vector configuration \(\mA\) is irreducible because \(\mA\) is projectively equivalent to \(\mV\) (by Lemma~\ref{lem:barthe}) and \(\mV\) is irreducible (by Lemma~\ref{lem:Virr}).
\end{proof}

\begin{proof}[Proof of Theorem~\ref{thm:fromPKtoV}]
	Let \((\mH, w) = \{(H_i, w_i)\}_{i=1}^n\) be a real PK arrangement in \(\R^d\) consisting of linear hyperplanes \(H_i \subset \R^d\) together with weights \(w_i \in (0,1)\) satisfying~\ref{it:PK1},~\ref{it:PK2}, and~\ref{it:PK3}. Let \(\mA\) be as in Lemma~\ref{lem:barthe}. By Lemma~\ref{lem:Avs}, \(\mA\) is an irreducible Euclidean \(\vee\)-system in \(\R^d\). The associated weighted arrangement \((\Delta_{\mA}, r_{\mA})\) satisfies: (i) \(\Delta_{\mA}\) is linearly equivalent to \(\mH\) by Lemma~\ref{lem:equivH}; and (ii) \(r_{\mA} = w\) by Lemma~\ref{lem:barthe} (ii). 
	
	Suppose now that \(\mB = \{\beta_i\}_{i=1}^n\) is a \(\vee\)-system in \(\R^d\) such that its associated weighted arrangement \((\Delta_{\mB}, r_{\mB})\) satisfies: (i) \(\Delta_{\mB}\) is linearly equivalent to \(\mH\); and (ii) \(r_{\mB} = r_{\mA} = w\). By Lemma~\ref{lem:linequivha}, (i) implies that \(\mB\) is projectively equivalent to \(\mA\). By item~\ref{it:V1'} of Lemma~\ref{lem:vstf}, \(\mB\) is a normalized tight frame for \(\R^d\). It then follows from uniqueness in Barthe's theorem (Proposition~\ref{prop:uniqueness}) that \(\mB\) is orthogonal projectively equivalent to \(\mA\).
\end{proof}

\section{Proofs of Corollary \ref{cor:bijection} and Corollary \ref{cor:classification}}\label{sec:corrolaries}

\begin{proof}[Proof of Corollary \ref{cor:bijection}]
By Theorem~\ref{thm:fromVtoPK}, the weighted arrangement associated with an
irreducible Euclidean \(\vee\)-system is a real PK arrangement. Moreover, by
Definitions~\ref{def:orthogonally-proj-equiv} and
\ref{def:linear-equivalence}, orthogonal projectively equivalent
\(\vee\)-systems determine linearly equivalent weighted arrangements.
Consequently, the assignment in the statement induces a well-defined map on
the corresponding equivalence classes.

This map is surjective by Theorem~\ref{thm:fromPKtoV}: given a real PK
arrangement \((\mH,w)\), that theorem produces an irreducible Euclidean
\(\vee\)-system \(\mA\) such that \(\Delta_{\mA}\) is linearly equivalent to
\(\mH\) and \(r_{\mA}=w\).

To prove injectivity, suppose that two irreducible Euclidean
\(\vee\)-systems \(\mA\) and \(\mA'\) determine the same equivalence class of
real PK arrangements. After relabelling, if necessary,
\(\Delta_{\mA'}\) is linearly equivalent to \(\Delta_{\mA}\) and
\(r_{\mA'}=r_{\mA}\). Thus both \(\mA\) and \(\mA'\) satisfy properties
\textup{(i)} and \textup{(ii)} of Theorem~\ref{thm:fromPKtoV} for the real
PK arrangement \((\Delta_{\mA},r_{\mA})\). The uniqueness statement in that
theorem implies that \(\mA\) and \(\mA'\) are orthogonal projectively
equivalent. Hence the map is injective, and therefore bijective.
\end{proof}

\begin{proof}[Proof of Corollary \ref{cor:classification}]
Let \(\mA\) be an irreducible Euclidean \(\vee\)-system composed of vectors of the same length, and let \((\Delta_{\mA},r_{\mA})\) be its associated weighted arrangement. By Theorem~\ref{thm:fromVtoPK}, \((\Delta_{\mA},r_{\mA})\) is a real PK arrangement. Since all the weights are equal, \(\Delta_{\mA}\) has the Hirzebruch property by \cite[Section~5.4]{panov}. By \cite[Theorem~1.1]{panov-real-hir}, \(\Delta_{\mA}\) is the mirror arrangement of the reflection group of the regular tetrahedron, cube, or icosahedron.

Conversely, for each of these three reflection arrangements, choose normal vectors of the same length and normalize them so that the sum of their squared lengths is \(3\). It is easy to check that the resulting collection of vectors is a Euclidean \(\vee\)-system. Moreover, for a fixed reflection arrangement its associated weighted arrangement is uniquely determined by this normalization. Hence Corollary~\ref{cor:bijection} shows that any irreducible Euclidean \(\vee\)-system with equal-length vectors and the same associated reflection arrangement is orthogonal projectively equivalent to this one.
\end{proof}

\section{Moduli spaces of \(\vee\)-systems, Theorem~\ref{thm:modulispace}}\label{sec:moduli}

This section identifies the moduli space of Euclidean \(\vee\)-systems in a fixed
projective class. In Proposition \ref{prop:msptf} we use Barthe's theorem, together with the uniqueness statement from Proposition~\ref{prop:uniqueness}, to parametrize normalized tight frames in a fixed projective class (up to orthogonal projective equivalence) by the relative interior of the associated base polytope. Proposition~\ref{prop:hirineq} then cuts out the Euclidean \(\vee\)-systems by the kernel of the Hirzebruch quadratic form.

\subsection{The algebraic variety of tight frames}
The set of all normalized tight frames of \(n\) vectors in \(\R^d\) has the natural structure of a real algebraic subvariety \(T_{n, d} \subset \R^{d\times n}\). We recall this following \cite[Chapter 7]{waldron}.

Let \((\alpha_i)_{i=1}^n\) be a sequence of vectors in \(\R^d\) and let
\[
A := [\alpha_1, \ldots, \alpha_n]
\]
be the \((d \times n)\)-matrix whose columns are the vectors \(\alpha_i\). Then \((\alpha_i)_{i=1}^n\) is a normalized tight frame for \(\R^d\) if and only if the matrix \(A\) satisfies
\begin{equation}\label{eq:mateq}
	A \, A^{\top} = \Id \,.
\end{equation}
Hence we identify the set of all normalized tight frames of \(n\) vectors in \(\R^d\) with 
\[
T_{n, d} = \{ A \in \R^{d \times n} \,|\, A\,A^{\top} = \Id \} \,.
\]
Thus \(T_{n,d}\) is a real algebraic subset of \(\R^{d \times n}\) cut out by degree two polynomials. It is also easy to see that \(T_{n, d}\) is a smooth compact submanifold. Indeed, Equation \eqref{eq:mateq} means that the rows of \(A\) are orthonormal, so \(T_{n,d}\) is homeomorphic to the \emph{Stiefel manifold}, i.e. the homogeneous space \(O(n)/O(n-d)\) consisting of orthonormal \(d\)-frames in \(\R^n\).

\subsection{Tight frames within a projective class}
Let \(\mA= \{\alpha_i\}_{i=1}^n\) be a vector configuration in \(\R^d\). We denote by \([\mA] \subset \R^{d\times n}\) the set of all vector configurations that are projectively equivalent to \(\mA\). Thus \([\mA]\) is the orbit of the matrix \(A = [\alpha_1, \ldots, \alpha_n]\) under the action of the group \(G= (\R^*)^n \times \GL(d)\) on \(\R^{d \times n}\) given by
\begin{equation}\label{eq:action}
((\lambda_1, \ldots, \lambda_n), F) \cdot A := [\lambda_1 F(\alpha_1), \ldots, \lambda_n F(\alpha_n)] 	
\end{equation}
where \(\lambda_i \in \R^*\) and \(F \in \GL(d)\).

Let \(K\) be the compact subgroup of \(G\) defined as \(K:= \{\pm 1\}^n \times O(d)\). The restriction of the action \eqref{eq:action} to \(K\) preserves the set of tight frames \(T_{n, d}\). Moreover, two vector configurations are orthogonal projectively equivalent if and only if they lie on the same \(K\)-orbit.

\begin{definition}
	The moduli space of normalized tight frames that are projectively equivalent to \([\mA]\) modulo orthogonal projective equivalence is defined as the set
	\begin{equation}
		\mathcal{T}_{[\mA]} := \left( T_{n,d}  \,\bigcap\, [\mA] \right) \,\big/\, K
	\end{equation}
	equipped with the quotient topology.
\end{definition}

\begin{proposition}\label{prop:msptf}
	Assume that \(\mA=\{\alpha_i\}_{i=1}^n\subset \R^d\) is an irreducible vector configuration. Let \(P(\mA)\) be the base polytope of \(\mA\). Then the map
	\[
	[\gamma_1,\ldots,\gamma_n]\longmapsto
	\left(|\gamma_1|^2,\ldots,|\gamma_n|^2\right)
	\]
	induces a homeomorphism
	\[
	\mT_{[\mA]}\cong \relin(P(\mA)).
	\]
\end{proposition}

\begin{proof}
For a matrix \(C=[\gamma_1,\ldots,\gamma_n]\in \R^{d\times n}\), define
\[
f(C)=\left(|\gamma_1|^2,\ldots,|\gamma_n|^2\right).
\]
The map \(f:\R^{d\times n}\to\R^n\) is continuous and invariant under the action
of \(K=\{\pm 1\}^n\times O(d)\). Hence, after restricting to
\(T_{n,d}\cap[\mA]\) and passing to the quotient, it induces a continuous map \(f: \mT_{[\mA]} \to \R^n\). By Theorem~\ref{thm:barthe} and Proposition~\ref{prop:uniqueness}
\[
f:\mT_{[\mA]}\longrightarrow \relin(P(\mA)).
\]
is a bijection.

We now need to prove that the inverse map is continuous.
To construct a continuous inverse, we recall some details of the proof of Barthe's theorem from \cite{aks}. 
Fix \(w\in\relin(P(\mA))\).
Following \cite[Equation (5)]{aks}, define \(\varphi:\R^n\to\R\) by
\[
\varphi(t)=\log\det R(t)- \sum_{i=1}^n w_i \, t_i,
\]
where
\[
R(t)=\sum_{i=1}^{n}e^{t_i}\hat{\alpha}_i\otimes\hat{\alpha}_i
\]
and \(\hat{\alpha}_i=\alpha_i/|\alpha_i|\). 
Note that, since \(\sum_{i=1}^n w_i=d\), we have \(\varphi(t+\lambda\mathbf 1)=\varphi(t)\)
for every \(\lambda\in\R\). We therefore restrict \(\varphi\) to the hyperplane
\[
\Sigma=\left\{t\in\R^n:\sum_{i=1}^n t_i=0\right\}.
\]
Since \(\mA\) is irreducible, the restriction of \(\varphi\) to \(\Sigma\) is
strictly convex. Indeed, \cite[Lemma~2.3]{aks} gives the convexity of
\(\log\det R(t)\), and strict convexity on \(\Sigma\) follows from \cite[Lemma~2.9]{aks} and the discussion
preceding it.
Moreover, since \(w\in\relin(P(\mA))\), \cite[Lemma~2.6]{aks}
implies that \(\varphi|_{\Sigma}\) has a minimizer. By strict convexity this
minimizer is unique; denote it by \(t_*=t_*(w)\).

Set
\[
T=R(t_*)^{-1/2}
\]
and define \(C=[\gamma_1,\ldots,\gamma_n]\) by
\[
\gamma_i
=
w_i^{1/2}
\frac{T\hat{\alpha}_i}{|T\hat{\alpha}_i|},
\qquad i=1,\ldots,n.
\]
By \cite[Proposition 2.7]{aks}, the matrix \(T\) puts the configuration
\(\{\hat{\alpha}_i\}_{i=1}^n\) in \emph{radial \(w\)-isotropic position} (\cite[Definition 1.2]{aks}). Equivalently,
\[
\sum_{i=1}^n w_i
\frac{T\hat{\alpha}_i}{|T\hat{\alpha}_i|}
\otimes
\frac{T\hat{\alpha}_i}{|T\hat{\alpha}_i|}
=
I_d.
\]
Therefore
\[
\sum_{i=1}^n \gamma_i\otimes\gamma_i=I_d.
\]
Thus \(C\) is a normalized tight frame. It is projectively equivalent to \(\mA\),
and by construction
\[
|\gamma_i|^2=w_i
\]
for every \(i\). This defines a map
\[
g:\relin(P(\mA))\longrightarrow T_{n,d}\cap[\mA],
\qquad
g(w)=C,
\]
such that \(f\circ g(w)=w\).

We now prove that \(g\) is smooth. 
By strict convexity of \(\varphi|_{\Sigma}\),
the minimizer \(t_*\in\Sigma\) is the unique critical point of \(\varphi|_{\Sigma}\). Moreover, the critical point \(t_*\) is non-degenerate. Hence the implicit function theorem implies that \(t_* = t_*(w)\) depends smoothly on \(w \in \relin(P(\mA))\).
Since \(R(t)^{-1/2}\) depends smoothly on \(t\), it follows that \(g\) is
smooth and therefore it defines a continuous inverse to \(f\).

Thus \(f:\mT_{[\mA]}\to\relin(P(\mA))\) is a continuous bijection with continuous
inverse, and hence is a homeomorphism.
\end{proof}

\subsection{Moduli spaces of Euclidean \(\vee\)-systems}\label{sec:pf-moduli-thm}
Let \(\mA= \{\alpha_i\}_{i=1}^n\) be an irreducible vector configuration in \(\R^d\). As before, we denote by \([\mA] \subset \R^{d \times n}\) the set of all vector configurations projectively equivalent to \(\mA\). 

\begin{notation}
Let \(V_{n, d} \subset \R^{d \times n}\) denote the set of all \(d \times n\) matrices whose columns form a \(\vee\)-system of \(n\) vectors in \(\R^d\). 
\end{notation}

Note that, by Lemma~\ref{lem:vstf}, \(V_{n,d} \subset T_{n,d}\). Moreover, \(V_{n,d}\) is invariant by the action of group \(K = \{\pm 1\}^n \times O(d)\) given by \eqref{eq:action}.

\begin{definition}
	The moduli space of Euclidean \(\vee\)-systems projectively equivalent to \(\mA\) up to orthogonal projective equivalence is defined as the set
	\[
	\msv = \left( V_{n,d}  \,\bigcap\, [\mA] \right) \,\big/\, K
	\]
	equipped with the quotient topology.
\end{definition}

Next, we need a simple lemma on quadratic forms.

\begin{lemma}\label{lem:quadratic-zero-locus}
Let \(Q\) be a quadratic form on \(\mathbb R^n\), and let \(H\subset \mathbb R^n\setminus 0\) be a hyperplane. Let \(U\subset H\) be an open subset, and suppose that \(Q(w)\leq 0\)
for all \(w\in U\). Then
\[
\{w\in U:Q(w)=0\}=U\cap\ker(Q),
\]
where \(\ker(Q)\) denotes the kernel, or radical, of \(Q\).
\end{lemma}

\begin{proof}  It is enough to prove that if \(w\in U\) and \(Q(w)=0\), then \(w\in\ker(Q)\), since the reverse implication follows from homogeneity.

Suppose that \(w\in U\) and \(Q(w)=0\). Since \(Q\) is homogeneous, there is an open subset \(\widetilde U\subset \mathbb R^n\) containing \(U\) on which \(Q\) is non-positive.  Therefore \(\nabla Q(w)=0\), which is equivalent to \(w\in\ker(Q)\).
\end{proof}

We can now prove the main result of this section.

\begin{proof}[Proof of Theorem~\ref{thm:modulispace}]
Note that \(\msv\) is naturally a subset of \(\mT_{[\mA]}\). Consider the homeomorphism \(f: \mT_{[\mA]} \to \relin(P(\mA))\) of Proposition~\ref{prop:msptf} given by \([\gamma_1, \ldots, \gamma_n] \mapsto (|\gamma_1|^2, \ldots, |\gamma_n|^2)\). It follows from Proposition~\ref{prop:hirineq} that a point \(x \in \mT_{[\mA]}\) belongs to \(\msv\) if and only if \(Q(f(x)) = 0\), where \(Q\) is the Hirzebruch quadratic form of \(\Delta_{\mA}\). 
    
It remains to identify the zero locus of \(Q\) inside \(\relin(P(\mA))\).
Since \(\mA\) is irreducible, Lemma~\ref{lem:basepolytope-dim} implies that \(\dim P(\mA)=n-1\).

Hence \(\relin(P(\mA))\) is open in the affine hyperplane \(\Sigma_d=\left\{w\in \mathbb R^n \mid \sum_{i=1}^n w_i=d\right\}\).
By Theorem~\ref{thm:barthe} , every \(w\in \relin(P(\mA))\) is realised as the squared-length vector of a normalized tight frame projectively equivalent to \(\mA\). Therefore,
by Proposition~\ref{prop:hirineq}, \(Q(w)\leq 0\)
for all \(w\in \relin(P(\mA))\). 
By Lemma \ref{lem:quadratic-zero-locus},
\[
\left\{w\in \operatorname{relin}(P(\mA))\mid Q(w)=0\right\}
=
\operatorname{relin}(P(\mA))\cap \ker(Q).
\]
Therefore, the map \(f\) restricts to a homeomorphism between \(\msv\) and \(\operatorname{relin}(P(\mA))\cap \ker(Q)\).
\end{proof}

\begin{remark}
	If \(d=2\) and \(\mH\) consists of \(n \geq 3\) lines through the origin in \(\R^2\), then the Hirzebruch quadratic form \(Q: \R^n \to \R\) of \(\mH\) is identically zero. In this case the moduli space of \(\vee\)-systems \(\msv\) agrees with the moduli space of tight frames \(\mT_{[\mA]}\), namely it is homeomorphic to \(\relin(P(\mA))\) where \(P(\mA)\) is the hypersimplex 
    \[
    \Delta(2, n) = \{w \in [0,1]^n \, \textup{ s.t. } \, \sum w_i = 2\} \,.
    \]
\end{remark}

\section{Euclidean \(\vee\)-systems and simplicial arrangements, Theorem~\ref{thm:simplicial}}\label{sec:simplicial}

In this section we prove that the arrangement associated with a Euclidean
\(\vee\)-system is simplicial. The argument is local around codimension-two
intersections of the arrangement. The two possible cases in the definition of a
\(\vee\)-system imply that adjacent facets (codimension $1$ faces) of a chamber meet at angle at most
\(\pi/2\): pairs give right angles, while planes containing at least three vectors
are controlled by the elementary geometry of tight frames in \(\R^2\). Thus the
associated arrangement is non-obtuse, and Coxeter's theorem on non-obtuse
spherical polytopes implies that it is simplicial.

\subsection{Tight frames on \(\R^2\)}

The next result is well-known and elementary. For completeness, we present a proof.

\begin{lemma}[{\cite[Exercise 2.9]{waldron}}]\label{lem:square-criterion}
	Let \((\alpha_j)_{j=1}^n\) be a sequence of vectors in \(\mathbb{R}^2\) with \(\alpha_j=(x_j,y_j)^\top\) and assume that at least one of the vectors \(\alpha_j\) is non-zero. Write \(z_j=x_j+iy_j\in\mathbb{C}\). Then \((\alpha_j)_{j=1}^n\) is a tight frame for \(\R^2\) if and only if
	\begin{equation}\label{eq:square-criterion}
	\sum_{j=1}^n z_j^2 = 0 \,.
	\end{equation}
\end{lemma}

\begin{proof}
	The matrix of the linear transformation \(\alpha_j \otimes \alpha_j\) is
	\[
	\alpha_j \, \alpha_j^\top
	=
	\begin{pmatrix}
	x_j^2 & x_jy_j \\
	x_jy_j & y_j^2
	\end{pmatrix}.
	\]
	By definition, the sequence \((\alpha_j)_{j=1}^n\) is a tight frame for \(\R^2\) if \(\sum_j \alpha_j \, \alpha_j^\top = \lambda \Id\). In terms of the vector components,
	\[
	\sum_{j=1}^n \left(x_j^2-y_j^2\right)=0
	\quad\text{and}\quad
	\sum_{j=1}^n x_jy_j=0.
	\]
	Since \(z_j^2=\left(x_j^2-y_j^2\right)+2ix_jy_j\), the above two equations are equivalent to the single Equation~\eqref{eq:square-criterion} in the complex plane.
\end{proof}

\begin{lemma}\label{lem:lines}
	Let \((\alpha_i)_{i=1}^n\) be a tight frame for \(\R^2\) consisting of \(n \geq 3\) pairwise non-proportional vectors. Then the lines spanned by the vectors \(\alpha_i\) divide the plane into angular sectors of angles \(<\pi/2\).
\end{lemma}

\begin{proof} 
    Suppose  by contradiction that the lines \(L_j\) spanned by the vectors \(\alpha_j=z_j\) do not divide the plane into angular sectors of angles \(< \pi/2\). Then, after a rotation, we can assume that all of the lines are contained in the union of the closed first and third quadrants. Then the points \(z_j^2\) lie all in the closed upper half plane \(\mathrm{Im}(z) \geq 0\). Since the number of points \(z_j^2\) is \(\geq 3\), and their arguments \(\mathrm{arg}(z_j^2)\) are pairwise distinct, at least one of them must lie on the open part \(\mathrm{Im}(z) > 0\), and therefore \(\mathrm{Im}(\sum_jz_j^2) > 0\). This contradicts \eqref{eq:square-criterion}.
\end{proof}

\subsection{Non-obtuse and simplicial arrangements}

Let \(\Delta\) be an arrangement of linear hyperplanes in \(\mathbb R^d\). A \emph{chamber} of \(\Delta\) is a connected component of
\[
\mathbb R^d \setminus \bigcup_{H \in \Delta} H .
\]
The closure of every chamber is a polyhedral cone.
A full-dimensional polyhedral cone \(C \subset \mathbb R^d\) is called \emph{simplicial} if it is generated by \(d\) linearly independent vectors. 

\begin{definition}
	Let \(\Delta\) be an arrangement of linear hyperplanes in \(\mathbb R^d\). We say that \(\Delta\) is \emph{simplicial} if the closure of every chamber is a simplicial cone.
\end{definition}

Now equip \(\R^d\) with its standard Euclidean inner product.

\begin{definition}
	Let \(\Delta\) be an arrangement of linear hyperplanes in \(\mathbb R^d\).
	We say that \(\Delta\) is \emph{non-obtuse} if, for every chamber \(U\) of \(\Delta\), and every pair of adjacent facets \(F_1, F_2\) of
	\(C=\overline U\), the dihedral angle between \(F_1\) and \(F_2\) is at most \(\pi/2\).
\end{definition}

\begin{lemma}\label{lem:nonobtuse}
	Any essential non-obtuse arrangement \(\Delta\) of linear hyperplanes in \(\mathbb R^d\) is simplicial.
\end{lemma}

\begin{proof}
Since \(\Delta\) is essential, every chamber closure is a full-dimensional pointed polyhedral cone. In particular, every chamber closure has at least \(d\) facets. Let \(C\) be the closure of a chamber of \(\Delta\). Its intersection with the unit sphere \(P = C \cap S^{d-1}\) is a convex spherical polytope.
The dihedral angles of \(P\) are exactly the dihedral angles of the cone \(C\). By a theorem of Coxeter, every convex spherical polytope whose dihedral angles are at most \(\pi/2\) is a simplex \cite[Theorem~1]{coxeter-reflections}.
\end{proof}

\subsection{Proof of Theorem~\ref{thm:simplicial}}

\begin{lemma}\label{lem:vsysnonobtuse}
	Let \(\mA = \{\alpha_i\}_{i=1}^n\) be a Euclidean \(\vee\)-system in \(\R^d\). Let \(\Delta_{\mA} = \{H_i\}_{i=1}^n\) be the associated hyperplane arrangement, where \(H_i = \alpha_i^{\perp}\). Then \(\Delta_{\mA}\) is non-obtuse.
\end{lemma}

\begin{proof}
	Let \(C\) be the closure of a chamber of \(\Delta_{\mA}\) and let \(F_1, F_2\) be two adjacent facets of \(C\). We want to show that the dihedral angle between \(F_1\) and \(F_2\) is at most \(\pi/2\). We divide the proof into two cases.
	
	\textbf{Case 1:} the intersection \(F_1 \cap F_2\) is contained in exactly two hyperplanes of \(\Delta_{\mA}\), say \(H_i\) and \(H_j\). In this case the normal vectors to the hyperplanes \(\alpha_i\) and \(\alpha_j\) are orthogonal, and hence the dihedral angle is \(\pi/2\).
	
	\textbf{Case 2:} the intersection \(F_1 \cap F_2\) is contained in at least three hyperplanes of \(\Delta_{\mA}\). Let \(S \subset \R^d\) be the linear span of \(F_1 \cap F_2\). 
    The dihedral angle between \(F_1\) and \(F_2\) is equal to the angle of the corresponding sector in  the orthogonal plane \(S^\perp\).
    The vectors \(\{\alpha_i \,\mid\, S \subset H_i\} = \{\alpha_i \,\mid\, \alpha_i \in S^\perp\}\) form a tight frame for the \(2\)-plane \(S^{\perp} \cong \R^2\).
    By Lemma~\ref{lem:lines}, the dihedral angle between \(F_1\) and \(F_2\) is \(< \pi/2\).
\end{proof}

We can now prove the main result of the section.

\begin{proof}[Proof of Theorem~\ref{thm:simplicial}]
    By Lemma~\ref{lem:associated-arrangement-essential} the associated arrangement \(\Delta_A\) is essential.
	By Lemma~\ref{lem:vsysnonobtuse} the arrangement \(\Delta_{\mA}\) is non-obtuse. By Lemma~\ref{lem:nonobtuse}, \(\Delta_{\mA}\) is simplicial.
\end{proof}

\begin{remark} The simplicial property of real PK line arrangements follows immediately from the existence of PK metrics on the complexification \cite{panov}. Indeed, the restriction of the PK metric to $\mathbb {RP}^2$ makes all chambers of the arrangement non-obtuse (see \cite[Theorem 4.1]{panov-real-hir} for more details). Theorem~\ref{thm:simplicial}, combined with Theorem~\ref{thm:fromPKtoV}, gives a new simpler proof of simpliciality, bypassing the theory of parabolic bundles \cite{mochizuki}.
\end{remark}

\section{Euclidean \(\vee\)-systems in \(\R^3\), Theorem~\ref{thm:classification3d}}\label{sec:3dclassification}

In the final section we apply the correspondence between Euclidean \(\vee\)-systems and real PK arrangements (Theorems~\ref{thm:fromVtoPK} and~\ref{thm:fromPKtoV}) to the rank-three classification problem. A central arrangement of planes in \(\R^3\) is projectivized to an
arrangement of lines in \(\RP^2\). We use the linear criterion of Corollary~\ref{cor:linear-pk} to test which of the known simplicial line arrangements arise from real PK arrangements (in dimension three, see also \cite{panov}). In the course of the computation, we need to solve systems of linear equations with integer coefficients in up to $37$ variables —the maximal number of lines in a sporadic simplicial arrangement. To streamline these calculations, we use SageMath.
Combining this calculation with the simpliciality Theorem~\ref{thm:simplicial} and Cuntz's completeness theorem for simplicial arrangements with at most \(27\) lines \cite{cuntz} gives the classification statement in Theorem~\ref{thm:classification3d}. 

\subsection{Linear equations for PK line arrangements}

Let \(\Delta = \{H_i\}_{i=1}^n\) be an arrangement of \(2\)-planes through the origin in \(\R^3\). We projectivize and think of \(\Delta\) as an arrangement of lines in \(\RP^2\). We denote by \(L_i = \mathbb{P}(H_i)\) the line \(L_i \subset \RP^2\) corresponding to the \(2\)-plane \(H_i \subset \R^3\).

Let \(B=(B_{ij})\) be the matrix associated with \(\Delta\) by Definition~\ref{def:Bmatrix}. Since \(d=3\), the system~\eqref{eq:linear-pk-system} becomes
\begin{equation}\label{eq:pksystem}
\sum_{j=1}^{n} B_{ij} w_j = 1 \quad \textup{for } i=1, \ldots, n \quad \textup{and} \quad \sum_{j=1}^{n} w_j = 3 \,.
\end{equation}

\begin{definition}\label{def:klt}
	Let \(\mathcal U \subset \R^n\) be the open set of all \(w \in (0,1)^n\) such that for every multiple point \(p\) of the arrangement of multiplicity \(\geq 3\) we have
	\begin{equation}\label{eq:klt}
		\sum_{i \,\mid\, p \in L_i} w_i < 2 \,.
	\end{equation}
\end{definition}

The following is the dimension-three specialization of Corollary~\ref{cor:linear-pk}; see also \cite[Section~5.3]{panov}.

\begin{lemma}\label{lem:dimasystem}
	The weighted arrangement \((\Delta, w)\) is PK if and only if \(w\) solves the system \eqref{eq:pksystem} and \(w \in \mathcal{U}\).
\end{lemma}

\subsection{Check of the PK condition on Gr\"unbaum's catalog}

The known simplicial line arrangements were first organised in Gr\"unbaum's catalog \cite{grunbaum}. The catalog contains three infinite families, denoted by Gr\"unbaum as \(R(0)\), \(R(1)\), and \(R(2)\). In the \(A(n,k)\)-notation of the catalog, the first family \(R(0)\) consists of the near-pencils \(A(n,0)\), \(n\geq 3\): these are arrangements of \(n\)
lines in which \(n-1\) lines pass through one point and the remaining line does not pass through that point.  The second family \(R(1)\) consists of the arrangements \(A(2m,1)\), \(m\geq 3\): these are obtained from a regular \(m\)-gon by taking the \(m\) lines containing its sides together with its \(m\) axes of symmetry.  The third family \(R(2)\) consists of the arrangements \(A(4m+1,1)\), \(m\geq 2\): these are obtained from \(A(4m,1)\) by adjoining the line at infinity.  In addition to these three infinite families, the updated Gr\"unbaum catalog \cite{cel} contains \(95\) known sporadic simplicial arrangements.
The catalog is known to be complete for arrangements with at most \(27\)
lines \cite{cuntz}, but it is not known whether the full list of sporadic examples is complete, or even whether it is finite.

Given a line arrangement \(\Delta = \{L_1, \ldots, L_n\}\) in \(\RP^2\), we say that \(w \in \R^n\) is a \emph{PK weight} for \(\Delta\) if the weighted arrangement \((\Delta, w)\) is PK.

\begin{proposition}\label{prop:computer}
	Out of the 3 infinite families and 95 sporadic examples of the known simplicial line arrangements, exactly 17 of them admit PK weights as detailed in Table~\ref{tab:pk-simplicial-arrangements} of the Appendix~\ref{app:tables}.
\end{proposition}

\begin{proof}
	Given a simplicial line arrangement \(\Delta\) as above,
	we use the criterion of Lemma~\ref{lem:dimasystem} to find the possible PK weights. That is, we first determine the affine space \(S \subset \R^n\) of solutions to the inhomogeneous system of \(n+1\) linear equations \eqref{eq:pksystem} and then we consider the intersection of \(S\) with the open set \(\mathcal{U}\) defined as the intersection of all \(w \in (0,1)^n\) such that the sum of the weights at the multiple points is \(< 2\), as in Definition~\ref{def:klt}.
	
	\emph{Infinite families.} All of the three infinite families admit solutions to the equations \eqref{eq:pksystem}. For the near-pencil \(A(n,0)\), if the first \(n-1\) lines meet, then
	\[
	S = \left\{  \sum_{i=1}^{n-1} w_i = 2 \, \textup{ and } \, w_n=1 \right\} \,.
	\]
	Clearly \(S \cap \mathcal{U} = \emptyset\), so none of these arrangement is PK. 
	
	Consider now the arrangements \(A(2m, 1)\) consisting of a regular \(m\)-gon and its axes of symmetry. For \(m \geq 5\) the system \eqref{eq:pksystem} has a unique solution with weights \(w_i =1/m\) for the sides and \(w_j=2/m\) for the axes. Clearly \(w \notin \mathcal{U}\), as the sum of the weights at the intersection of the axes is \(2\). The cases \(m=3\) and \(m=4\) are dealt with separately. Direct calculation shows that the affine subspace \(S\) is \(3\)-dimensional for \(A(6,1)\) and \(2\)-dimensional for \(A(8,1)\) and the intersection with the open set \(\mathcal{U}\) is a \(3\)-simplex and a \(2\)-simplex respectively.
	
	Consider the arrangements \(A(4m+1,1)\), obtained from \(A(4m, 1)\) by adding a line at infinity. For \(m \geq 4\) the system \eqref{eq:pksystem} admits a unique solution \(w\) with weights same as for \(A(4m,1)\) and weight zero for the line at infinity; clearly \(w \notin \mathcal{U}\). The cases \(m=2\) and \(3\) are dealt separately by direct calculation using a computer (see below). For \(m=2\) the space \(S\) is \(3\)-dimensional and the intersection \(S\cap \mathcal{U}\) is the interior of a \(3\)-dimensional polyhedron combinatorially isomorphic to a triangular bipyramid. For \(m=3\), \(S\) is one-dimensional and the intersection \(S \cap \mathcal{U}\) is an open interval.
	
	\emph{Sporadic examples.} To analyse the 95 sporadic examples we use the \texttt{cn\_hyperarr} package for SageMath from \cite{cel}. This package contains a database for the wiring diagrams for all the known simplicial line arrangements. From the wiring diagrams we can calculate the matroid (or multiple point intersections) of the arrangement. Then we solve the system of equations \eqref{eq:pksystem} and for the arrangements that have solutions \(S\) we check if the affine subspace \(S\) intersects the open set \(\mathcal{U}\) using linear programming methods of SageMath. The jupyter notebook with the calculations is available in 
    \begin{center}
        \url{https://github.com/MdeBorbon/PK-SIMPLICIAL.git}
    \end{center}
    All computations in the notebook are performed exactly over \(\mathbb{Q}\); in particular, the feasibility tests do not rely on floating-point approximations or numerical tolerances.
\end{proof}

\subsection{Classification of 3d Euclidean \(\vee\)-systems of order at most 27}

\begin{proof}[Proof of Theorem~\ref{thm:classification3d}]
	Let \(\mA\) be an irreducible Euclidean \(\vee\)-system in \(\R^3\) with
	\(|\mA|\leq 27\), and let \((\Delta_{\mA},r_{\mA})\) be its associated weighted
	arrangement. By Theorem~\ref{thm:simplicial}, the arrangement \(\Delta_{\mA}\)
	is simplicial. By Theorem~\ref{thm:fromVtoPK}, the weighted arrangement
	\((\Delta_{\mA},r_{\mA})\) is a real PK arrangement. Thus \(\Delta_{\mA}\)
	is a simplicial arrangement of \(|\mA|\leq 27\) lines in \(\RP^2\) admitting
	PK weights.

	By Cuntz's completeness theorem \cite{cuntz}, every simplicial arrangement of
	at most \(27\) lines in \(\RP^2\) occurs in the catalog considered above. Applying
	Proposition~\ref{prop:computer}, the arrangements admitting
	PK weights in this range are precisely those appearing in
	Table~\ref{tab:pk-simplicial-arrangements} of the Appendix~\ref{app:tables} with at most \(27\) lines.

	The correspondence displayed in Table~\ref{tab:pk-veselov} of the Appendix~\ref{app:tables} identifies these
	arrangements with the rank-three \(\vee\)-systems in the Schreiber--Veselov
	catalog \cite{schreiberveselov}. Hence the projective class of \(\mA\), and
	therefore \(\mA\) itself up to orthogonal projective equivalence by
	Theorem~\ref{thm:fromPKtoV}, is one of the \(\vee\)-systems listed in
	\cite{schreiberveselov}. This proves the theorem.
\end{proof}

\subsection{Remarks}

\begin{remark} According to \cite[Proposition 5.3]{panov}, any PK arrangement \(\{L_1, \ldots, L_n\}\subset \CP^2\) has the following property: for every \(L_j\), there exists a point in \(\CP^2\) that belongs to all lines \(L_k\) such that \(L_k\cap L_j\) is a double point of the arrangement. This criterion is quite effective for showing that a given line arrangement is not PK. For example, roughly half—precisely \(46\)—of the sporadic simplicial arrangements in Gr\"unbaum's list violate this property and can therefore be shown not to be PK without any calculation, simply by inspecting a picture of the arrangement. This property also holds for complex \(\vee\)-systems of rank $3$, as can be proved by exactly the same method as in \cite[Proposition 5.3]{panov}.
\end{remark}

\begin{remark}
The linear equations~\eqref{eq:linear-pk-system} also arise more generally for complex hyperplane arrangements supporting flat logarithmic connections of the type studied in \cite[Equation~(1.1)]{dbp-pkc}. This includes the connections associated with complex \(\vee\)-systems \cite{feigin-veselov-geometryVsystem} and Dunkl systems \cite{chl}. More precisely, if the arrangement consists of \(n\) hyperplanes in \(\C^d\), then the traces \(w_i\) of the residues along the hyperplanes satisfy the \(n\) homogeneous linear equations displayed in \cite[Equation~(8.11)]{dbp-pkc}, one for each hyperplane. Together with the normalization \(\sum_{i=1}^n w_i=d\), these are exactly the equations in the overdetermined system~\eqref{eq:linear-pk-system}. The requirement that this system admit a solution rules out the vast majority of combinatorial types of arrangements.
\end{remark}

\appendix
\section{Simplicial line arrangements admitting PK weights and the corresponding rank-three $\vee$-systems}\label{app:tables}

\begin{table}[ht]
	\centering
	\begin{tabular}{c c}
		\toprule
		Gr\"unbaum & PK weights \\
		\midrule
		\(A(6,1)\)  & \(3\)-simplex \\
		\(A(7,1)\)  & \(2\)-simplex \\
		\(A(8,1)\)  & \(2\)-simplex \\
		\(A(9,1)\)  & triangular bipyramid \\
		\(A(10,2)\) & interval \\
		\(A(10,3)\) & point \\
		\(A(11,1)\) & interval \\
		\(A(13,1)\) & interval \\
		\(A(13,2)\) & interval \\
		\(A(13,3)\) & point \\
		\(A(15,1)\) & point \\
		\(A(16,3)\) & point \\
		\(A(17,2)\) & point \\
		\(A(17,4)\) & point \\
		\(A(19,1)\) & point \\
		\(A(19,3)\) & point \\
		\(A(31,1)\) & point \\
		\bottomrule
	\end{tabular}
	\caption{Simplicial arrangements admitting PK weights.}
	\label{tab:pk-simplicial-arrangements}
\end{table}

\begin{table}[ht]
	\centering
	\begin{tabular}{c c}
		\hline
		Gr\"unbaum  & \(\vee\)-system \\
		\hline
		\(A(6,1)\)  & \(A_3(c_1,c_2,c_3)\) \\
		\(A(7,1)\)  & \(D_3(t,s)\) \\
		\(A(8,1)\)  & \((E_6,A_3)\) \\
		\(A(9,1)\)  & \(B_3(c_1,c_2,c_3,\gamma)\) \\
		\(A(10,2)\) & \((AB_4(t),A_1)_2\) \\
		\(A(10,3)\) & \((E_6,A_1^3)\) \\
		\(A(11,1)\) & \((AB_4(t),A_1)_1\) \\
		\(A(13,1)\) & \(G_3(t)\) \\
		\(A(13,2)\) & \(F_3(t)\) \\
		\(A(13,3)\) & \((E_7,A_1^2\times A_2)\) \\
		\(A(15,1)\) & \(H_3\) \\
		\(A(16,3)\) & \((E_8,A_1\times A_4)\) \\
		\(A(17,2)\) & \((E_8,A_1^2\times A_3)\) \\
		\(A(17,4)\) & \((E_8,A_2\times A_3)\) \\
		\(A(19,1)\) & \((E_8,A_2^2\times A_1)\) \\
		\(A(19,3)\) & \((E_8,A_1^3\times A_2)\) \\
		\(A(31,1)\) & \((H_4,A_1)\) \\
		\hline
	\end{tabular}
	\caption{The simplicial arrangements admitting PK weights and the corresponding rank-three \(\vee\)-systems in Schreiber--Veselov catalog \cite{schreiberveselov}.}
	\label{tab:pk-veselov}
\end{table}

\begin{figure}[htbp]
    \centering
    \includegraphics[scale=0.5]{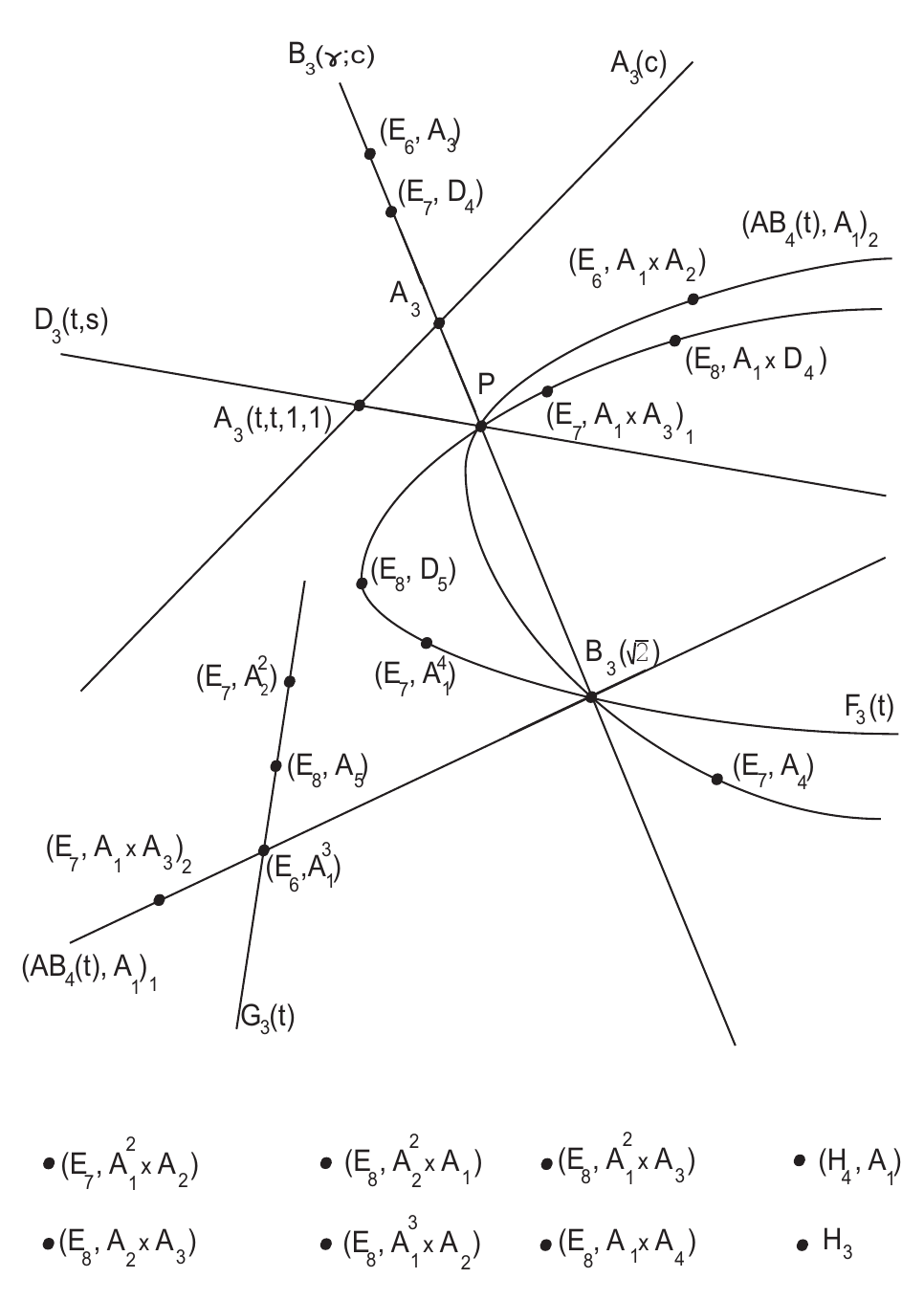}
    \caption{Schematic map of all known rank-three irreducible \(\vee\)-systems from Feigin--Veselov \cite{feigin-veselov-geometryVsystem}.}
    \label{fig:FV_list}
\end{figure}

\end{document}